\documentclass[10pt,reqno]{amsart}

\usepackage{latexsym,amsfonts,amssymb,exscale,enumerate,comment}
\usepackage{amsmath,amsthm,amscd}

\newcommand{\MH}[1]{\textcolor[rgb]{.8,.3,1}{MH: #1}}
\renewcommand{\a}{\alpha}
\newcommand{\Q}{\mathbb{Q}}

\usepackage{url}
\usepackage[bookmarks=true,%
    colorlinks=true,%
    linkcolor=blue,%
    citecolor=blue,%
    filecolor=blue,%
    menucolor=blue,%
    urlcolor=blue,%
    breaklinks=true]{hyperref}

\input xy
\usepackage[all]{xy}
\xyoption{line}
\xyoption{arrow}
\xyoption{color}
\SelectTips{cm}{}

\usepackage{tikz}
\usetikzlibrary{decorations.markings}
\usetikzlibrary{decorations.pathreplacing}
\tikzstyle directed=[postaction={decorate,decoration={markings,
    mark=at position #1 with {\arrow{>}}}}]

\newcommand{\hackcenter}[1]{
 \xy (0,0)*{#1}; \endxy}

\tikzset{->-/.style={decoration={
  markings,
  mark=at position #1 with {\arrow{>}}},postaction={decorate}}}

\tikzset{middlearrow/.style={
        decoration={markings,
            mark= at position 0.5 with {\arrow{#1}} ,
        },
        postaction={decorate}
    }
}

\usepackage{graphicx}
\usepackage{color}

\theoremstyle{plain}
\newtheorem{theorem}{Theorem}

\newtheorem{corollary}[theorem]{Corollary}
\newtheorem{proposition}[theorem]{Proposition}
\newtheorem{lemma}[theorem]{Lemma}

\theoremstyle{definition}
\newtheorem{example}[theorem]{Example}
\newtheorem{definition}[theorem]{Definition}

\theoremstyle{definition}
\newtheorem{remark}[theorem]{Remark}

\numberwithin{equation}{section}
\numberwithin{theorem}{section}

\newcommand{\maps}{\colon}
\newcommand{\und}[1]{\underline{#1}}

\newcommand{\refequal}[1]{\xy {\ar@{=}^{#1}
(-1,0)*{};(1,0)*{}};
\endxy}

\renewcommand{\to}{\rightarrow}

\def\Res{{\mathrm{Res}}}
\def\Ind{{\mathrm{Ind}}}

\def\smod{{\mathrm{-smod}_{\mathrm{pol}}}}  

\def\mf{\mathfrak}

\numberwithin{equation}{section}
\let\hat=\widehat
\let\tilde=\widetilde

\let\epsilon=\varepsilon

\usepackage{bbm}
\def\C{{\mathbb{C}}}
\def\N{{\mathbbm N}}

\def\Z{{\mathbbm Z}}
\def\H{{\mathcal{H}}}

\def\cal#1{\mathcal{#1}}%
\def\1{\mathbbm{1}}%
\def\nn{\notag}

\def\la{\langle}
\def\ra{\rangle}

\usepackage{bbm}

\def\cal#1{\mathcal{#1}}

\newcommand\nc{\newcommand}
\nc\rnc{\renewcommand}
\nc\Kar{\operatorname{Kar}}
\nc\End{\operatorname{End}}
\nc{\im}{\operatorname{im}}

\newcommand{\scs}{\scriptstyle}

\nc\Sym{\operatorname{Sym}}

\newcommand{\Sq}{{\rm Sq}}

\allowdisplaybreaks

\nc\Omit[1]{}

\nc\sfk{\mathsf{k}}
\nc\IQ{\mathbb{Q}}
\nc\IZ{\mathbb{Z}}

\nc\NH[1]{\mathsf{NH}_{#1}}

\nc\dd[1]{\partial_{#1}}
\nc\xv[1]{x_{#1}}
\nc\wv[1]{\omega_{#1}}
\nc\swv[1]{\omega_{#1}^{\scs{\sfs{}}}}
\nc\schwv[1]{\vartheta_{#1}}
\nc\dschwv[1]{\vartheta^{*}_{#1}}
\nc\ewv[1]{\mathsf{e}^{\omega}_{#1}}
\nc\hwv[1]{\mathsf{h}^{\omega}_{#1}}
\nc\schwvt[1]{\tilde{\vartheta}_{#1}}
\nc\schwvr[1]{{\vartheta}^{(r)}_{#1}}
\nc\schu[1]{\mathfrak{s}_{#1}}
\nc\Sp[1]{\mathfrak{S}_{#1}}
\nc\dSp[1]{\mathfrak{S}^{*}_{#1}}

\nc\symL[1]{\Lambda_{#1}}
\nc\Pol[1]{\mathsf{P}_{#1}}

\nc\ext{^{\scs{\mathsf{ext}}}}
\nc\extP[1]{\mathsf{P}\ext_{#1}}
\nc\extL[1]{\Lambda\ext_{#1}}
\nc\extNH[1]{\mathsf{NH}\ext_{#1}}

\nc\Sgp{\mathsf{S}}

\nc\ten{\otimes}
\nc\Wedge[1]{\bigwedge[#1]}
\nc\lp{\left(}
\nc\rp{\right)}
\nc\wdt[1]{\widetilde{#1}}
\nc\id{\operatorname{id}}
\nc\aand{\qquad\mbox{and}\qquad}

\rnc\S{\mathcal{S}}

\nc\ef{\mathsf{e}}
\nc\hf{\mathsf{h}}
\nc\pf{\mathsf{p}}

\newcommand{\Lw}{L^w}

\nc\sfu[1]{\mathsf{u}_{#1}}
\nc\sfr[1]{\mathsf{r}_{#1}}
\nc\sfs[1]{\mathsf{s}_{#1}}
\nc\sfv[1]{\mathsf{v}_{#1}}
\nc\Smx{\mathbb{S}}
\nc\sfvt[1]{\widetilde{\mathsf{v}}_{#1}}
\nc\sfD[1]{\mathsf{D}_{#1}}
\nc\sfM[1]{\mathsf{M}_{#1}}

\title{
A DG-extension of symmetric functions arising from higher representation theory
}

\begin{document}
\setcounter{tocdepth}{1}

\author{Andrea Appel}
\email{appelmb@gmail.com}
\address{Department of Mathematics\\ University of Southern California \\ Los Angeles, CA}

\author{Ilknur Egilmez}
\email{egilmez@usc.edu}
\address{Department of Mathematics\\ University of Southern California \\ Los Angeles, CA}

\author{Matthew Hogancamp}
\email{hogancam@usc.edu}
\address{Department of Mathematics\\ University of Southern California \\ Los Angeles, CA}

\author{Aaron D. Lauda}
\email{lauda@usc.edu}
\address{Department of Mathematics\\ University of Southern California \\ Los Angeles, CA}

%

\begin{abstract}
We investigate analogs of symmetric functions arising from an extension of the nilHecke algebra defined by Naisse and Vaz.  These extended symmetric functions form a subalgebra of the polynomial ring tensored with an exterior algebra.  We define families of bases for this algebra and show that it admits a family of differentials making it a sub-DG-algebra of the extended nilHecke algebra.  The ring of extended symmetric functions equipped with this differential is quasi-isomorphic to the cohomology of a Grassmannian.  We also introduce new deformed differentials on the extended nilHecke algebra that when restricted makes extended symmetric functions quasi-isomorphic to $GL(N)$-equivariant cohomology of Grassmannians.
\end{abstract}

\maketitle

\tableofcontents

One of the most fundamental objects in higher representation theory is the nilHecke algebra~\cite{Lau1,Rou2,KL1}.  This object is the most basic ingredient in categorified quantum groups and is intimately related to the geometry of flag varieties and Grassmannians~\cite{KoKu,Lau3}.  The nilHecke algebra admits a faithful action on the polynomial ring, further relating it to the combinatorics of symmetric functions and Schubert polynomials.

The categorification, or higher representation theory, perspective has demonstrated that extensions or alternative categorifications of quantum groups often have parallel implications in geometry and combinatorics.
As an example,  one motivation for studying the odd (spin/super) nilHecke algebra~\cite{KW1,Wang,EKL,KKO} was an attempt to supply a representation theoretic explanation for the appearance of ``odd Khovanov homology" -- a distinct link homology theory with similar properties to Khovanov homology.
The odd nilHecke algebra shared many of the relationships of the usual nilHecke algebra, including connections to a new noncommutative Hopf algebra of symmetric functions with strikingly similar combinatorics~\cite{EK}.  The odd nilHecke algebra gave ``odd" noncommutative analog of the cohomology of Grassmannians and Springer varieties~\cite{LauR,EKL}.  All of these developments grew out of the discovery of an odd analog of the nilHecke algebra.

Recently, Naisse and Vaz \cite{Vaz} have introduced an extension of the nilHecke algebra $\extNH{n}$ that we refer to as the {\em extended nilHecke algebra}.   This algebra arose in the study of a fundamental issue in higher representation theory.  The problem  was the fact biadjointness for $\cal{E}$ and $\cal{F}$ in the definition of  categorified quantum groups~\cite{KL3,Rou2} implied that it was only possible to categorify finite dimensional modules; in particular, categorical analogs of Verma modules were inaccessible within the existing theory.  Naisse and Vaz overcame this issue in the case of $\mf{sl}_2$, by omitting the biadjointness condition, enhancing the nilHecke algebra to the extended nilHecke algebra, and altering the main $\mf{sl}_2$-relation to a short exact sequence, rather than a direct sum decomposition.  This work allowed for the first categorification of Verma modules and may be an indication of the way forward in higher representation theory.

Given the importance of the extended nilHecke algebra in categorifying Verma modules, this article investigates the combinatorial implications of this algebra.  We define analogs of symmetric functions $\extL{n}$ arising from the extended nilHecke algebra that we call {\em extended symmetric functions}.  We construct families of bases for these algebras and investigate their combinatorial properties.  Extending the work of Naisse and Vaz, we show that the ring $\extL{n}$ admits a family of differentials $d_N$ such that $(\extL{n},d_N)$ is a sub-DG-algebra of the extended nilHecke algebra. Additionally, we show that the extended nilHecke algebra with its differential $d_N$ is isomorphic to the Koszul complex associated to a regular sequence of central elements in $\NH{n}$.  Restricting to $(\extL{n},d_N)$ gives a DG-algebra which is quasi-isomorphic to the cohomology ring of a Grassmannian $Gr(n,N)$.  The algebra $\extL{n}$ has been independently discovered by Naisse and Vaz using different techniques~\cite{Vaz2}.

Our work facilitates an explicit realization of the extended nilHecke algebra $\extNH{n}$ as a matrix ring of size $n!$ over its center,  the ring of extended symmetric functions.  This identifies the ring $\extL{n}$ with the center of the DG-algebra $\extNH{n}$.   The importance of the explicit isomorphism as a matrix ring over a positively graded algebra is that it allows us to define primitive idempotents decomposing the identity $1 \in \extNH{n}$.  This implies $\extNH{n}$ has a unique graded indecomposable supermodule up to isomorphism and grading shift.  Using this fact, we prove that the family of extended nilHecke algebras $\extNH{n}$, taken for all $n\geq0$, categorifies the bialgebra corresponding to the positive part $\mathbf{U}^+(\mf{sl}_2)$ of the quantized universal enveloping algebra of $\mf{sl}_2$, suggesting that the extended nilHecke algebra likely fits into a similar extension of KLR-algebras categorifying $\mathbf{U}^+(\mf{g})$ for symmetrizable Kac-Moody algebras.

We also define new deformed differentials $d_N^{\Sigma}$ on $\extNH{n}$ in section~\ref{sec:deformed-diff}.  The deformed differentials also restrict to $\extL{n}$ and the resulting cohomology of $(\extL{n},d_N^{\Sigma})$ is generically isomorphic to the $GL(N)$-equivariant cohomology of a Grassmannian.

Let us point out more clearly the relation between our work and \cite{Vaz}.  In \emph{loc.~cit.}, Vaz-Naisse define bigraded algebras $\Omega_k$ ($k\in\Z_{\geq 0}$) and bigraded bimodules ${}_{\Omega_{k+1}}\cal{F}_{\Omega_k}$, ${}_{\Omega_k}\cal{E}_{\Omega_{k+1}}$.  These bimodules generate a 2-category which categorifies the Verma module for quantum $\mf{sl}_2$ with generic highest weight.  In this context, the extended nilHecke $\extNH{n}$ algebra arises as the ring of bimodule endomorphisms of $\cal{F}^{\otimes n}$, or equivalently $\cal{E}^{\otimes n}$.  Our work provides an idempotent decomposition of $\cal{E}^{\otimes n}$ (respectively $\cal{F}^{\otimes n}$) as a direct sum of $n!$ copies with shifts of a bimodule $\cal{E}^{(n)}$ (respectively $\cal{F}^{(n)}$), thereby paving the way for a ``thick calculus'' version of the Vaz-Naisse 2-category, similar to what was accomplished in \cite{KLMS}.  In this context, the ring of extended symmetric functions appears as the ring of endomorphisms of $\cal{E}^{(n)}$ and $\cal{F}^{(n)}$.  It occurs that the resulting endomorphism ring is isomorphic to $\Omega_n$, so that  ${}_{\Omega_{n+k}}\cal{F}^{(n)}_{\Omega_k}$ and ${}_{\Omega_{k-n}}\cal{E}^{(n)}_{\Omega_{k}}$ may be more appropriately referred to as \emph{trimodules} over $(\Omega_{k\pm n},\Omega_n,\Omega_k)$.  We remark that all of the above is compatible with the differentials $d_N$ in the appropriate sense.  See \ref{ss:cat} for more.

Finally, we mention an interpretation of the algebraic structures appearing in this subject in terms of Khovanov-Rozansky homology, both the doubly graded $\mf{sl}_N$ version \cite{KhR} and the triply graded HOMFLY-PT version \cite{KR08b}.  The cohomology rings of Grassmannian $Gr(k,N)$ can be thought of as the $\mf{sl}_N$ homology of the $k$-colored unknot \cite{Wu,Yoshi}, while the ring of extended symmetric functions $\extL{k}$ can be thought of as the HOMFLY-PT homology of the $k$-colored unknot \cite{WW09}.  The Koszul differential $d_N$ considered here and in \cite{Vaz} is then a special case of Rasmussen's $\mf{sl}_N$ differential \cite{Rasmussen}.  We expect that the trimodules ${}_{\Omega_{n+k}}\cal{F}^{(n)}_{\Omega_k}$ and ${}_{\Omega_{k-n}}\cal{E}^{(n)}_{\Omega_{k}}$ appear in this setting as the homologies of certain MOY diagrams, namely the colored theta graphs.  This is likely to be related to the point of view adopted by Vaz and Naisse in ~\cite{Vaz3}.

\subsection*{Acknowledgements}  We would like to thank David Rose for helpful conversations on deformed differentials and Weiqiang Wang for pointing out the connection between extended symmetric functions and the work of Solomon.   All authors would like to acknowledge partial support by NSF grant DMS-1255334.  This project grew out of a NSF CAREER sponsored vertically integrated Categorification learning seminar at the University of Southern California.

\section{The nilHecke algebra}

Many of our constructions for the extended nilHecke algebra build off of results for the usual nilHecke algebra and its action on polynomials.  Here we recall the relevant results.

\subsection{The definition}

Recall the nilHecke algebra $\NH{n}$ defined by generators $x_i$ for $1 \leq i \leq n$ and $\partial_j$ for $1 \leq j \leq n-1$ and relations
\begin{equation}\label{eq:nilHecke}
 \begin{array}{ll}
 x_i x_j =   x_j x_i , &  \\
   \partial_i x_j = x_j\partial_i \quad \text{if $|i-j|>1$}, &
   \partial_i\partial_j = \partial_j\partial_i \quad \text{if $|i-j|>1$}, \\
  \partial_i^2 = 0,  &
   \partial_i\partial_{i+1}\partial_i = \partial_{i+1}\partial_i\partial_{i+1},  \\
   x_i \partial_i - \partial_i x_{i+1}=1,  &   \partial_i x_i - x_{i+1} \partial_i =1.
  \end{array}
\end{equation}
It is not hard to prove that these relations imply
\begin{equation}\label{eq:ind-dotslide}
  \partial_i x_{i}^{a+1} - x_{i+1}^{a+1}\partial_i= \hf_{a}(x_{i},x_{i+1})
 = x_{i}^{a+1}\partial_i - \partial_i x_{i+1}^{a+1}.
\end{equation}
for all $a>0$.

Given any element $w \in \Sgp_n$ and a reduced decomposition $w =s_{i_1} \dots s_{i_m}$ into simple transpositions we write $\partial_{w}:= \partial_{i_1} \dots \partial_{i_m}$.  The axioms ensure this definition does not depend on the choice of reduced expression.  We write $w_0$ for the longest word in the symmetric group $\Sgp_n$ and $\partial_{w_0}$ for the corresponding product of divided difference operators.

The algebra $\NH{n}$ acts on the polynomial ring $\Pol{n} :=\IQ[\xv{1},\dots, \xv{n}]$
with $x_i$ acting by multiplication by $x_i$ and $\partial_i \maps \Pol{n} \to \Pol{n}$
given by divided difference operators
\begin{equation}
 \partial_i := \frac{1-s_i}{x_i-x_{i+1}}.
\end{equation}

We recall several important facts relating to the nilHecke algebra and its action on polynomials.
\begin{itemize}
  \item The ring of symmetric functions can be realized strictly in terms of the divided difference operators
\[
 \Lambda_n \; :=\; \Z[x_1,\dots, x_n]^{\Sgp_n} \;=\; \bigcap_{j=1}^{n-1} \ker \partial_i
 \;=\; \bigcap_{j=1}^{n-1}\im \partial_i.
\]

  \item The additive basis of $\Lambda_n$ given by Schur functions $\mathfrak{s}_{\lambda}$ can be defined using the nilHecke algebra action on polynomials via
  \[
  \mathfrak{s}_{\lambda}:= \partial_{w_0}(\und{x}^{\delta+\lambda}) := \partial_{w_0}(x_1^{n-1 + \lambda_1} x_2^{n-2 + \lambda_2} \dots x_n^{0+\lambda_n}),
  \]
for $\lambda=(\lambda_1,\dots,\lambda_n)$ a partition with $n$ parts.

  \item For $w \in \Sgp_n$ define the {\em Schubert polynomials} of Lascoux and
Sch\"{u}tzenberger~\cite{Las} as
\begin{equation}
\Sp{w}(x) = \partial_{w^{-1}w_0}x^{\delta}
\end{equation}
where $w_0$ is the permutation of maximal length and
$x^{\delta}=x_1^{a-1}x_2^{a-2}\cdots x_{a-1}$.  In case $w=1\in \Sgp_n$, we have $\Sp{\id} = \partial_{w_0}(x^\delta)=1$.

\item   We have
\begin{equation}
\im \partial_{w_0} = \symL{n}\subset \Pol{n}.
\end{equation}
Indeed, if $f\in \symL{n}$, then $f = f \partial_{w_0}(x^{\delta}) = \partial_{w_0}(f x^\delta)$ since divided difference operators are $\symL{n}$-linear.  Conversely, if $f\in
\im \partial_{w_0}$, then $\partial_i(f)=0$ for $i=1,\ldots,n-1$, hence $f\in \Pol{n}^{\Sgp_n}$.

\item The polynomial ring $\Pol{n}$ is a free module over $\Lambda_n$ of rank $n!$~\cite[Proposition 2.5.5 and 2.5.5]{Man}.  In particular, multiplication in $\Pol{n}$ induces a ring isomorphism $\Pol{n}\simeq \H_n\ten\Lambda_n$
where $\H_n$ is equivalently the abelian subgroup spanned by either of the sets
$\left\{ \Sp{w} \mid w \in \Sgp_n   \right\}$ or 
$\left\{  x_1^{i_1} \dots x_n^{i_n} \mid 0 \leq i_k \leq n-k .  \right\}$.
\end{itemize}
The last statement allows us to identify $\End_{\Lambda_n}(\Pol{n})$ as the matrix ring of size $n!$ with coefficients in the ring $\Lambda_n$.  The ring $\Pol{n}$ is graded with $\deg(x_i) =2$.  Taking grading into account, it follows that there is an isomorphism of graded rings $\End_{\Lambda_n}(\Pol{n}) \cong {\rm Mat}( (n)^!_{q^2}; \Lambda_n)$, where
$(n)_{q^2}^! =q^{n(n-1)/2}[n]!$ are the symmetric quantum factorials~\cite[Proposition 3.5]{Lau1}.

The action of $\NH{n}$ on $\Pol{n}$ defines a graded ring homomorphism
\[
\gamma \maps \NH{n} \to {\rm Mat}( (n)^!_{q^2}; \Lambda_n).
\]
It was shown in \cite[Proposition 3.5]{Lau1} that $\gamma$ is an isomorphism of graded rings.  We recall an alternative proof from \cite{KLMS} Section 2.5 that we translate into algebraic language from the so-called ``thick calculus".

For any composition $\mu=(\mu_1,\dots, \mu_n)$ write $\und{x}^{\mu}:= x_1^{\mu_1} x_2^{\mu_2} \dots x_n^{\mu_n}$.  We write $\und{x}^{\delta}:= x_1^{n-1}x_2^{n-2}\dots x_n^0$.
The set of sequences
\begin{equation}
  \Sq(n) := \{
  \und{\ell} = \ell_1 \dots \ell_{n-1} \mid 0 \leq \ell_{\nu} \leq \nu, \;\; \nu =  1,2, \dots n-1
  \}
\end{equation}
has size $|\Sq(n)|=n!$.  Let $|\und{\ell}|=\sum_{\nu} \ell_{\nu}$, and set $\hat{\ell_j}=j-\ell_j$.
Define a composition with $n$-parts by
\begin{equation}
  \hat{\und{\ell}}=(0, \hat{\ell}_1,\dots,\hat{\ell}_{n-1})=  \left(0, 1-\ell_1, 2-\ell_2, \cdots,   n-1-\ell_{n-1} \right).
\end{equation}

Let $\ef_r^{(a)}$ denote the $r$th elementary symmetric polynomial in $a$ variables.   The {\em standard elementary monomials} are given  by
\begin{equation}
  \ef_{\und{\ell}} := \ef_{\ell_1}^{(1)}\ef_{\ell_2}^{(2)} \dots \ef_{\ell_{a-1}}^{(a-1)}.
\end{equation}

Define elements in $\NH{n}$ by
\begin{equation}
  \sigma_{\und{\ell}} := \ef_{\und{\ell}} \partial_{w_0}, \qquad \quad \lambda_{\und{\ell}} := (-1)^{\hat{\und{\ell}}}\;\und{x}^{\delta}\partial_{w_0}\und{x}^{\hat{\und{\ell}}}.
\end{equation}

\begin{theorem}[\cite{KLMS}] \label{thm_Ea} \hfill

\begin{enumerate}
\item For all $\ell,\ell'$ in $Sq(n)$, $\lambda_{\und{\ell'}} \cdot \sigma_{\und{\ell}} =\delta_{\und{\ell},\und{\ell'}}\; x^{\delta} \partial_{w_0}$.

\item The set
$
\left\{ \;
    \lambda_{\und{\ell}} \sigma_{\und{\ell}} \in   \Sq(n)
\;\right\}
$
form a complete set of mutually orthogonal primitive idempotents in $\NH{n}$.

\item The identity element $1 \in \NH{n}$ decomposes as
\begin{equation}
  1 = \sum_{\und{\ell}} (-1)^{\hat{\und{\ell}}}\ef_{\und{\ell}}\partial_{w_0} \und{x}^{\hat{\und{\ell}}}.
\end{equation}

\item  Enumerate the rows and columns of $n! \times n!$-matrices by the elements $\ell \in \Sq(n)$.  There is an isomorphism of graded algebras
\begin{equation} \label{eq:nil-to-matrix}
{\rm Mat}\left(
  (n)^!_{q^2}, \Lambda_{n}    \right) \longrightarrow \NH{n}
\end{equation}
sending an element $x \in \extL{n}$ in the $(\und{\ell}, \und{\ell'})$ entry to the element $\sigma_{\und{\ell}}x\lambda_{\und{\ell'}}$.
\end{enumerate}
\end{theorem}

The nilHecke algebra is the simplest example of a KLR-algebra, corresponding to the Lie algebra $\mf{sl}_2$.  The results above are critical in the categorification of   positive parts of quantized universal enveloping algebras via KLR-algebras~\cite{KL1,KL2,Rou2}.  Another important construction from categorified representation theory is the so-called {\em cyclotomic quotients} of KLR-algebras.  These are used to categorify irreducible representations of $\mathbf{U}_q(\mf{g})$.

For each $N>1$ define the cyclotomic ideal of $\NH{n}$ as the two sided ideal generated by $x_1^N$,
\begin{equation} \label{eq:cyclotomic}
  I_N := \la x_1^N\ra.
\end{equation}
We define the cyclotomic quotient by $\NH{n}^N:= \NH{n}/ I_N$.
We have the following results.
\begin{itemize}
  \item The isomorphism $\gamma$ from \eqref{eq:nil-to-matrix} induces an isomorphism~\cite[Proposition 5.3]{Lau3}
\begin{equation}
     {\rm Mat}\left(
  (n)^!_{q^2}, H^{\ast}(Gr(n,N))    \right) \longrightarrow \NH{n}^N
\end{equation}
where  $H^{\ast}(Gr(n,N))$ is the cohomology ring of the Grassmannian of complex $n$-planes in $\C^N$.
  \item The categories of graded projective modules over $\bigoplus_n \NH{n}^N$ categorify~\cite{LV,KK,Web5} the irreducible $\mathbf{U}_q(\mf{sl}_2)$ representation $V_N$ of highest weight $N$.
\end{itemize}

\section{The extended nilHecke algebra}

\subsection{The definition}

The extended nilHecke algebra $\extNH{n}$, first defined in ~\cite{Vaz}, is a graded superalgebra with even generators $x_1,\cdots,x_n$ and $\dd{1}, \cdots, \dd{n-1}$, and odd generators $\wv{1},\cdots, \wv{n}$ satisfying equations \eqref{eq:nilHecke} and the  following relations
\begin{equation*}
\begin{aligned}
&x_{i}\wv{j} = \wv{j}x_{i}, \hspace{1.5cm} \wv{i}\wv{j}= - \wv{j}\wv{i},\\
&\dd{i}\wv{j} = \wv{j}\dd{i} - \delta_{ij}\wv{i+1}(x_{i+1}\dd{i} - \dd{i}x_{i+1}).
\end{aligned}
\end{equation*}
For each fixed integer $k$ the algebra $\extNH{n}$ admits a $\Z$-grading with
\begin{equation}
  \deg(x_i) = 2, \qquad \deg(\partial_i) = -2, \qquad \deg(\omega_k) = -2k-2.
\end{equation}

%
%

For each $w\in \Sgp_n$ fix a reduced expression.  A basis for the superalgebra $\extNH{n}(m)$ is given in~\cite{Vaz} by the set of elements
\begin{equation}
\left\{ x_1^{i_1}\cdots x_n^{i_n}\wv{1}^{\alpha_1}\cdots \wv{n}^{\alpha_n}\dd{w} \right\}
\end{equation}
 $i_k \in \mathbb{N}$ and $\alpha_l \in \{0,1\}$.

\begin{remark}
  In \cite{Vaz} they consider an additional grading for their application to categorical Verma modules.  Here we ignore this additional grading.
\end{remark}

\subsection{Action on polynomials}

Define the {\em extended polynomial ring}
\begin{equation}
\extP{n}=\IQ[\xv{1},\dots, \xv{n}]\ten\Wedge{\wv{1},\dots, \wv{n}}.
\end{equation}
There is an action of $\extNH{n}$ on $\extP{n}$ defined by letting
$\xv{i}$ and $\wv{i}$ act by left multiplication and letting $\partial_i$ act by {\em extended divided difference operators}
\begin{equation*}
\dd{i}(1) = 0, \qquad \dd{i}(\wv{j}) = -\delta_{ij}\wv{j+1}, \qquad
\dd{i}(\xv{j}) =
\begin{cases}
1  & \text{if } j=i, \\
-1 & \text{if } j=i+1, \\
0  & \text{otherwise.}
\end{cases}
\end{equation*}
These operators are extended to arbitrary polynomials by the rule
\begin{equation}\label{eq:modLeib}
\dd{i}(fg) = \dd{i}(f)g + f\dd{i}(g) - (\xv{i}-\xv{i+1})\dd{i}(f)\dd{i}(g)
\end{equation}
for all $f,g \in \IQ[\xv{1},\dots, \xv{n}]\ten\Wedge{\wv{1},\dots, \wv{n}}$.

\subsection{Differentials} \label{subsec:diff}

Recall that a \emph{differential graded algebra} (or DG-algebra) is  a $\Z$-graded unital algebra $A$ with
$d \maps A \to A$ which is degree -1 satisfying
\begin{equation}
  d^2=0, \qquad d(ab) = d(a)b +(-1)^{\deg(a)\deg(b)}ad(b), \quad d(1) =0.
\end{equation}

A \emph{left DG-module} $M$ is a graded left $A$-module with differential $d_M \maps M_i \to M_{i-1}$ such that for all $a\in A$, $m\in M$,
\begin{equation}
  d_M(am) = d(a)m + (-1)^{\deg(a)}a d_M(m).
\end{equation}

For each $N>0$, define a differential $d_N$ on $\extNH{n}$ by
\begin{equation}
  d_N(x_i) =0, \qquad  d_N(\partial_i) =0, \qquad d_N(\omega_i) = (-1)^i \hf_{N-i+1}(\und{x}_i),
\end{equation}
where $\und{x}_i$ denotes the set of variables $\{x_1,x_2,\dots, x_i\}$.  Note the ordinary nilHecke algebra $\NH{n}$ is in the kernel of this differential for all $N$.  Furthermore, $d_N(\omega_1)=-x_1^N$.  By \cite[Proposition 2.8]{Lau4} it follows $d_N(\omega_j)$ is contained in the cyclotomic ideal $I_N:= \la x_1^N\ra$ from \eqref{eq:cyclotomic}.

\begin{theorem}[\cite{Vaz} Proposition 8.3]
  The DG-algebra $\left(\extNH{n}, d_N\right)$ is quasi-isomorphic to the cyclotomic quotient of the nilHecke algebra $\NH{n}^N := \NH{n}/I_N$.
\end{theorem}


\section{The ring of extended symmetric polynomials}

\subsection{Definition}

\subsubsection{Preliminary definition}

The action of $\extNH{n}$ on the extended polynomial ring
$\extP{n}=\IQ[\xv{1},\dots, \xv{n}]\ten\Wedge{\wv{1},\dots, \wv{n}}$
gives rise to a homomorphism
\[
\extNH{n}\to\End_{\IQ}\lp\extP{n}\rp.
\]
By analogy with the case of symmetric polynomials, we define the ring of extended
symmetric polynomials $\extL{n}$ as
\[
\extL{n}=\bigcap_{i=1}^{n-1}\mathsf{ker}\,\dd{i}=\bigcap_{i=1}^{n-1}\mathsf{im}\,\dd{i}.
\]

\subsubsection{Action of the symmetric group on $\extP{n}$}

The standard action of the symmetric group $\Sgp_n$ on the polynomial ring
$\Pol{n}=\IQ[\xv{1},\dots,\xv{n}]$ lifts to an action on $\extP{n}$. Namely, one
sets
\[
s_i(\xv{j})=\xv{s_i(j)}
\qquad\mbox{and}\qquad
s_i(\wv{j})=\wv{j}+\delta_{ij}(\xv{j}-\xv{j+1})\wv{j+1}
\]
for any $1\leq i\leq n-1$, $1\neq j\leq n$, and extends it to $\extP{n}$ by
$s_i(fg)=s_i(f)s_i(g)$
for any $f,g\in\extP{n}$. With respect to this action, the operators $\dd{i}$
coincide with the standard divided difference operators:
\begin{equation}\label{eq:dd}
\dd{i}=\frac{\id-s_i}{\xv{i}-\xv{i+1}}~.
\end{equation}
In particular, \eqref{eq:modLeib} reduces to the standard Leibniz rule
for divided difference operators
\begin{equation}
\dd{i}(fg)=\dd{i}(f)g+s_i(f)\dd{i}(g).
\end{equation}
\Omit{
\begin{proof}
Set $\wdt{\dd{i}}=\frac{\id-s_i}{\xv{i}-\xv{i+1}}$.
One check immediately that $\dd{i}$ and $\wdt{\dd{i}}$
coincide on the generators of $\extP{n}$. Moreover, they both satisfy
the modified Leibniz rule \eqref{}. Namely,
\begin{align*}
\wdt{\dd{i}}(fg)=&\frac{fg-s_i(fg)}{\xv{i}-\xv{i+1}}=\\
=&\wdt{\dd{i}}(f)g+s_i(f)\wdt{\dd{i}}(g)=\\
=&\wdt{\dd{i}}(f)g+f\wdt{\dd{i}}(g)-(\xv{i}-\xv{i+1})\wdt{\dd{i}}(f)\wdt{\dd{i}}(g)
\end{align*}
\end{proof}
}
It follows that $\extL{n}$ coincides with the subalgebra
of $\Sgp_n$--invariants $\extL{n}=\lp\extP{n}\rp^{\Sgp_n}$.

We now provide an explicit description of $\extL{2}$ and $\extL{3}$.
The general case is discussed in \ref{ss:sizeext} and \ref{ss:basis}.

\begin{remark}
The algebra $\extP{n}$ is endowed with another, more natural action of the symmetric
group (which on the other hand does not extend to an action of $\extNH{n}$).
Namely, for any $w\in\Sgp_n$, one can set $w(\wv{i})=\wv{w(i)}$.
The corresponding subalgebra of $\Sgp_n$--invariants is described by Solomon
in ~\cite{Solomon}, see also~\cite[Chapter 22]{Kane}. In Section \ref{s:Sol-thm}, we discuss the connection between these two actions and their invariants.
\end{remark}

\subsubsection{Case $n=2$}

The algebra $\extP{2}$ is a free module of rank $4$ over $\Pol{2}$. Then
it is easy to see that an element
\[
v=a+b\wv{1}+c\wv{2}+d\wv{1}\wv{2}\in\extL{2}
\]
if and only if
\[
a=s_1(a)
\quad
b=s_1(b)
\quad
c=s_1(c)+s_1(b)(\xv{1}-\xv{2})
\quad
d=s_1(d)
\]
or equivalently
$a,d\in\symL{2}$,
$\dd{1}(c)\in\symL{2}$,
and $b=\dd{1}(c)$.
The general solution to the equation $\dd{1}(c)\in\symL{2}$
has the form $c=Af+B$,
where $f\in\symL{2}$, $A,B\in\Pol{2}$, and
\[
\dd{1}(A)=1\aand\dd{1}(B)=0.
\]
Therefore for any choice of $A$ such that $\dd{1}(A)=1$,
$v\in\extL{2}$ if and only if $v$ has the form
\[
v=a+b(\wv{1}+A\wv{2})+c\wv{2}+d\wv{1}\wv{2}
\]
for some $a,b,c,d\in\symL{2}$. In particular, $\extL{2}$ is
a free module of rank $4$ over $\symL{2}$ with basis $\{1,\wv{1}+A\wv{2}, \wv{2}, \wv{1}\wv{2}\}$,
where $A$ is any solution of $\dd{1}(A)=1$.
Particular choices of $A$ are $\{\xv{1}, -\xv{2}, \frac{1}{2}(\xv{1}-\xv{2})\}$.

\subsubsection{Case $n=3$}

The algebra $\extP{3}$ is a free module of rank $8$ over $\Pol{3}$. Then
\[
v=a+b\wv{1}+c\wv{2}+d\wv{3}+e\wv{1}\wv{2}+f\wv{1}\wv{3}+g\wv{2}\wv{3}+h\wv{1}\wv{2}\wv{3}\in\extL{3}
\]
if and only if 
$a\in\symL{3}$, $b=\dd{1}\dd{2}(d)$, $c=\dd{2}(d)$, 
$e=\dd{2}\dd{1}(g)$, $f=\dd{1}(g)$, $h\in\symL{3}$, and
\[
\begin{array}{rcr}
\dd{1}(d)=0  		& & \dd{1}\dd{2}(d)\in\symL{3}\\
\dd{2}(g)=0 	&&  \dd{2}\dd{1}(g)\in\symL{3}.
\end{array}
\]
It is easy to show that the general solution of the system $\dd{1}(d)=0, \dd{1}\dd{2}(d)\in\symL{3}$
has the form
\[
d=Af_1+Bf_2+f_3,
\]
where $f_1,f_2, f_3\in\symL{3}$ and $A,B\in\Pol{3}$ are any solution of
\[
\begin{array}{rr}
\dd{1}(A)=0, & \quad		\dd{1}(B)=0, \\
\dd{1}\dd{2}(A)=1, &\quad  \dd{2}(B)=1.
\end{array}
\]
Similarly for $g$.
\Omit{the general solution of the system $\dd{2}(g)=0, \dd{1}\dd{2}(g)\in\symL{3}$
has the form
\[
g=Cg_1+Dg_2+g_3,
\]
where $g_1,g_2, g_3\in\symL{3}$ and $D,E\in\Pol{3}$ are any solution of
\[
\begin{array}{rr}
\dd{2}(C)=0, & \quad		\dd{1}(D)=1, \\
\dd{2}\dd{1},(C)=1 &\quad  \dd{2}(D)=0.
\end{array}
\]
}
We conclude that $\extL{3}$ is a free module over $\symL{3}$ of rank $8$
with basis
\[
\begin{array}{ccccc}
1 &  \wv{1}+\dd{2}(A)\wv{2}+A\wv{3} & \wv{2}+ B\wv{3} & \wv{3} & \\
  &\wv{1}\wv{2}+\dd{1}(C)\wv{1}\wv{3}+ C\wv{2}\wv{3} & \wv{1}\wv{3}+ D\wv{2}\wv{3} & \wv{2}\wv{3} &\wv{1}\wv{2}\wv{3}
\end{array}
\]
where $A,B,C,D\in\Pol{3}$ are any solution of
\[
\begin{array}{rrrr}
\dd{1}(A)=0 & 		\dd{1}(B)=0 & \dd{2}(C)=0 & 		\dd{1}(D)=1 \\
\dd{1}\dd{2}(A)=1 &  \dd{2}(B)=1 & \dd{2}\dd{1}(C)=1 &  \dd{2}(D)=0
\end{array}
\]
Particular choices of solutions of the above system are $A\in \{x_1x_2, x_3^2\}$, $B\in \{x_1+x_2, -x_3\}$, $C =x_1^2$, and $D =x_1$.

\subsection{The size of extended symmetric functions}\label{ss:sizeext}

We now discuss the general case for $n\geq3$.

\subsubsection{Notations}

For any binary sequence $\alpha\in\IZ_2^n$, set $\wv{\alpha}=\wv{1}^{\alpha_1}\cdots\wv{n}^{\alpha_n}$.
Then 
\[
\extP{n}=\bigoplus_{\alpha\in\IZ_2^n}\Pol{n}\cdot\wv{\alpha}.
\]
The action of $\Sgp_n$ is concisely described by the formula
\[
s_i(\wv{\alpha})=\wv{\alpha}+\delta_{\alpha_i, 1}\delta_{\alpha_{i+1},0}(\xv{i}-\xv{i+1})\wv{s_i(\alpha)}.
\]
For $k=1,\dots, n-1$ and $\alpha\in\IZ_2^n$, set
\begin{align*}
I_k=&\{\alpha\in\IZ_2^n\;|\; \alpha_k=1, \alpha_{k+1}=0\},\\
J_k=&\{\alpha\in\IZ_2^n\;|\;\alpha_k=0, \alpha_{k+1}=1\}=s_k(I_k),\\
D_{\alpha}=&\{k\;|\; \alpha\in J_k\},
\end{align*}
so that, in particular,
\[
s_k(\wv{\alpha})=
\left\{
\begin{array}{ccc}
\wv{\alpha} & \mbox{if} & \alpha\not\in I_k~,\\
\wv{\alpha}+(\xv{i}-\xv{i+1})\wv{s_i(\alpha)} & \mbox{if} & \alpha\in I_k~.
\end{array}
\right.
\]
For $k=0,1,\dots, n$, let $(\IZ_2^n)_k$ be the subset of strings of length $k$
\[
(\IZ_2^n)_k=\{\alpha\in\IZ_2^n\;|\;|\alpha|=\sum_{i=1}^n\alpha_i=k\}
\]
endowed with the following partial ordering $\prec$.
We say that $\alpha\prec\beta$ if there exists a
sequence in $(\IZ_2^n)_k$
\[\alpha_1=\alpha, \alpha_2, \dots, \alpha_m=\beta\]
where $m>1$ and for any $i=1,\dots, m-1$, $\alpha_i\in I_{r}$ and $\alpha_{i+1}\in J_{r}$
for some $r$.
Let $\tau^{(k)}, \lambda^{(k)}$ be, respectively, the highest and lowest element in
$((\IZ_2^n)_k, \prec)$,i.e. $\tau_i^{(k)}=0$ if and only if $i<n-k+1$ and $\lambda^{(k)}_i=0$
if and only if $i>k$.

\subsubsection{Grassmannian permutations}\label{sss:fund-perm}
A {\em Grassmannian permutation} $w$ is a permutation with a unique descent.  In other words there exists $k\in \{1,\ldots,n-1\}$ such that $w(i) < w(i+1)$ if $i \neq k$ and
$w(k) > w(k+1)$.\\

The Grassmannian permutations with descent $n-k$ are in canonical bijection with elements in $(\IZ_2^n)_k$, as we now describe.   Let $\a\in (\IZ_2^n)_k$ be given.  Let $1\leq \sfv{1}<\cdots <\sfv{n-k}\leq n$ be the indices such that $\a_{\sfv{1}}=\cdots=\a_{\sfv{n-k}}=0$, and let $1\leq \sfu{1}<\cdots<\sfu{k}\leq n$ be the indices such that $\a_{\sfu{1}}=\cdots=\a_{\sfu{k}}=1$.  Define $\sigma_{\alpha}\in\Sgp_n$ 
by
\[
\sigma_{\alpha}(i)=
\left\{
\begin{array}{ccc}
\sfv{i} & \mbox{if} & 1\leq i\leq n-k  \\
\sfu{i-n+k} & \mbox{if} &n-k+1\leq i\leq n
\end{array}
\right.
\]
More concisely, $\sigma_\a$ is the unique minimal length permutation which sends
\[
\tau^{(k)}=(\:\underbrace{0,\ldots,0}_{n-k},\underbrace{1,\ldots,1}_k\:) \mapsto \a.
\]
In particular, $\sigma_{\tau^{(k)}}=\id$.  Note that $\sigma_\a$ is a minimal length
representative of a coset in $\Sgp_n/\Sgp_{n-k}\times\Sgp_k$.\\

For every $\alpha\in(\IZ_2^n)_k$, $\alpha\neq\tau^{(k)}$,
$\sigma_{\alpha}$ has a unique descent at $n-k$, and it is therefore Grassmannian.  Conversely every Grassmannian permutation arises in this way.\\

\subsubsection{Lehmer codes and partitions}
Recall that the Lehmer code of a permutation $w$ is the composition $\Lw= (\Lw_1, \dots, \Lw_n)$, where \[
\Lw_i = \#\{ i< j \maps w(j) < w(i)\}.
\]
We write $\lambda(w)$ for the partition obtained by sorting $\Lw$ into decreasing order.
In particular, the Lehmer code of the Grassmannian permutation $\sigma_{\alpha}$, $\alpha\in(\IZ_2^n)_k$,
is given by
\begin{align*}
L_i^{\alpha}=
\left\{
\begin{array}{ccl}
 m &\mbox{if}& \sfu{m}-m<i\leq\sfu{m+1}-(m+1),\\
 0  & \mbox{if} & \sfu{m+1}-(m+1)<i.
 \end{array}
 \right.
\end{align*}
More concretely, if $1\leq i\leq n-k$, $L_i^{\alpha}$ is the number of ones which appear to the left of the $i$-th zero of $\a$, and $L_i^{\alpha}=0$ otherwise.
In particular, $\Lw_1 \leq \dots \leq \Lw_{n-k}$, and $\Lw_i=0$ for $i>n-k$.
The partition corresponding to $\sigma_{\alpha}$ is then
\[
\lambda_{\alpha}:=\lambda(\sigma_{\alpha})=(m^{\sfr{m}})_{m=k,\dots,1}.
\]
where $\sfr{m}=\sfu{m+1}-\sfu{m}-1$ for every $m=0,\dots, k$ (we impose $\sfu{0}=0, \sfu{k+1}=n+1$).
Notice that $\lambda_{\alpha}$ has at most $n-k$ non zero terms. In fact,
one sees immediately that the biggest possible size of the tableau of shape 
$\lambda_{\alpha}$,
$\alpha\in(\IZ_2^n)_k$, is $(n-k)\times k$.
The conjugate partition is $\lambda'_{\alpha}=(\lambda'_j)_{j=1,\dots, k}$
\[
\lambda'_j=\sum_{m= j}^k\sfr{m}=n+1-\sfu{j}-(k-j+1)=n-k-\sfu{j}+j.
\]

\subsubsection{Examples}
For any $1\leq j<k\leq n$, set $c_{[j,k]}=s_j\cdots s_{k-1}$ 
and $c^{(k)}=c_{[k,n]}\cdots c_{[2,n-k+2]}\cdot c_{[1, n-k+1]}$. We sometimes write $c_{[j]}:= c_{[j,n]}$.
It may be helpful to visualize these elements
$$c^{(k)} \;\; = \;\;
   \hackcenter{\begin{tikzpicture}[scale=0.9]
    \draw[thick] (0,0) .. controls ++(0,.4) and ++(0,-.4) .. (3,2);
    \draw[thick] (1.5,0) .. controls ++(0,.4) and ++(0,-.4) .. (4.5,2);
    \node at (1,.25) {$\cdots$};
    \draw[thick] (2.25,0) .. controls ++(0,.4) and ++(0,-.4) .. (0,2);
    \draw[thick] (3,0) .. controls ++(0,.4) and ++(0,-.4) .. (.75,2);
    \draw[thick] (4.5,0) .. controls ++(0,.4) and ++(0,-.4) .. (2.25,2);
    \node at (3.5,.25) {$\cdots$};
    \node at (0,-.25) {$\scs 1$};
    \node at (1.4,-.25) {$\scs n-k$};
    \node at (2.4,-.25) {$\scs n-k+1$};
    \node at (4.5,-.25) {$\scs n$};
\end{tikzpicture}}
$$
where diagrams are read from bottom to top.
Then it is easy to see that
\[
c^{(k)}(\tau^{(k)})=\lambda^{(k)}
\]
and, for any $\alpha\in(\IZ_2^n)_k$, $\sigma_{\alpha}$ is a subword of $c^{(k)}$.

\subsubsection{Main result}

The rest of this section is devoted to prove the following

\begin{theorem}\label{thm:ext-sym}\hfill
\begin{enumerate}[(i)]
\item The ring of extended symmetric polynomials $\extL{n}$ is a free module over $\symL{n}$
of rank $2^n$.
\Omit{
\item For each $k\in \{0,1,\ldots,n\}$ and each $\a\in (\IZ_2^n)_k$, choose a polynomial $p_\a\in \Pol{n}^{\Sgp_{n-k}\times\Sgp_k}$ such that
\begin{equation}\label{eq:palphaNormalization}
\partial_{\sigma_\a}(p_\a)=1 \ \ \ \ \ \ \ \ (\Pol{n}^{\Sgp_{n-k}\times\Sgp_k}),
\end{equation}
and define
\[
\swv{\alpha}(p_\a):=\wv{\alpha}+\sum_{\beta\succ\alpha}\dd{\sigma_{\beta}}(p_{\alpha})\cdot\wv{\beta}
\]
Then as $\symL{n}$--modules
\[
\extL{n}\simeq\bigoplus_{\alpha\in\IZ_2^n}\symL{n}\cdot\swv{\alpha}(p_\a)
\]
Moreover every $\symL{n}$--basis of $\extL{n}$ has this form.
}
\item For any collection of polynomials $\{p_{\alpha}\}_{\alpha\in\IZ_2^n}$ satisfying
\begin{equation}\label{eq:palphaNormalization}
p_{\alpha}\in\Pol{n}^{\Sgp_{n-|\alpha|}\times \Sgp_{|\alpha|}}
\aand
\dd{\sigma_{\alpha}}p_{\alpha}=1
\end{equation}
there is an isomorphism of $\symL{n}$--modules
\[
\extL{n}\simeq\bigoplus_{\alpha\in\IZ_2^n}\symL{n}\cdot\swv{\alpha}(p_\a)
\quad\mbox{where}\quad
\swv{\alpha}(p_\a):=\wv{\alpha}+\sum_{\beta\succ\alpha}\dd{\sigma_{\beta}}(p_{\alpha})\cdot\wv{\beta}
\]
\item Multiplication in $\extL{n}$ induces a ring isomorphism
\[
\extL{n}\simeq\symL{n}\ten\Wedge{\swv{1}, \dots, \swv{n}}.
\]
\item Multiplication in $\extP{n}$ induces a ring isomorphism
$\extP{n}\simeq \H_n\ten\extL{n}$, where $\H_n\subset \Pol{n}$ is the subspace 
spanned by either of the sets $\left\{ \Sp{w} \mid w \in \Sgp_n   \right\}$
or  $\left\{  x_1^{i_1} \dots x_n^{i_n} \mid 0 \leq i_k \leq n-k .  \right\}$.
This gives rise to a canonical ring isomorphism
\[
\End_{\extL{n}}(\extP{n})\simeq\mathsf{Mat}(n!, \extL{n}).
\]
\end{enumerate}
\end{theorem}

\begin{remark}
In \ref{ss:basis} we construct examples of $p_\a\in \Pol{n}^{\Sgp_{n-|\a|}\times\Sgp_{|\a|}}$ satisfying (\ref{eq:palphaNormalization}), for each $\a$.
\end{remark}

The proof is carried out in \ref{sss:char}--\ref{sss:end-pf-ext-sym}.

\subsubsection{First characterization of $\extL{n}$}\label{sss:char}

\begin{proposition}\label{prop:characterization1}
Let $v=\sum_{\alpha}f_\alpha\wv{\alpha}\in\extP{n}$.
The following are equivalent.
\begin{enumerate}[(i)]
\item $v\in\extL{n}$
\item For every $i=1,\dots, n-1$,
\begin{equation}\label{eq:sys-1}
\dd{i}(f_{\alpha})=
\left\{
\begin{array}{ccc}
0 & \mbox{if} & \alpha\not\in J_i,\\
f_{s_i(\alpha)} & \mbox{if} & \alpha\in J_i.
\end{array}
\right.
\end{equation}
\item For every $\alpha\in\IZ_2^n$,
\begin{equation}\label{eq:sys-2}
\dd{i}(f_{\alpha})=
\left\{
\begin{array}{ccc}
0 & \mbox{if} & i\not\in D_{\alpha},\\
f_{s_i(\alpha)} & \mbox{if} & i\in D_{\alpha}.
\end{array}
\right.
\end{equation}
\end{enumerate}
\end{proposition}

\begin{proof}
Clearly, $(ii)$ and $(iii)$ are equivalent. Now, let $v=\sum_{\alpha}f_{\alpha}\wv{\alpha}$,
$f_{\alpha}\in\Pol{n}$. For every $i=1,\dots, n-1$,
\begin{align*}
s_i(v)=&\sum_{\alpha\in\IZ_2^n}s_i(f_{\alpha})+\sum_{\alpha\in I_i}(\xv{i}-\xv{i+1})s_i(f_{\alpha})\wv{s_i(\alpha)}=\\
=&\sum_{\alpha\not\in J_i}s_i(f_{\alpha})\wv{\alpha}+
\sum_{\alpha\in J_i}\lp s_i(f_{\alpha})+s_i(f_{s_i(\alpha)})(\xv{i}-\xv{i+1})\rp\wv{\alpha}
\end{align*}
Therefore $v\in\extL{n}$ if and only if, for every $i=1, \dots, n-1$,
\[
\dd{i}(f_{\alpha})=
\left\{
\begin{array}{ccc}
0 & \mbox{if} & \alpha\not\in J_i,\\
s_i(f_{s_i(\alpha)}) & \mbox{if} & \alpha\in J_i.
\end{array}
\right.
\]
Finally, one observes that for every $\alpha\in J_i$, 
$s_i(\alpha)\not\in J_i$. Therefore,
$s_i(f_{s_i(\alpha)})=f_{s_i(\alpha)}$
and $(i)$ is equivalent to $(ii)$.
\end{proof}

\subsubsection{Simplification}

The system of equations \eqref{eq:sys-2} preserves $|\alpha|$, i.e.
there are $n+1$ independent sets of equations,
for $k=0,1,\dots, n$,
\begin{equation*}
\forall\alpha\in(\IZ_2^n)_k
\qquad
\dd{i}(f_{\alpha})=
\left\{
\begin{array}{ccc}
0 & \mbox{if} & i\not\in D_{\alpha},\\
f_{s_i(\alpha)} & \mbox{if} & i\in D_{\alpha}
\end{array}
\right.
\end{equation*}
Let $\tau^{(k)}, \lambda^{(k)}$ be, as before, the highest and lowest element in
$(\IZ_2^n)_k$ with respect to $\prec$.
Then it follows from \eqref{eq:sys-2} that $f_{\lambda^{(k)}}\in\symL{n}$ and,
for every $\alpha\in(\IZ_2^n)_k$,
\[
f_{\alpha}=\dd{\sigma_{\alpha}}(f_{\tau^{(k)}})
\]
In particular, any solution of \eqref{eq:sys-2}
is determined by the elements $f_{\tau^{(k)}}\in\Pol{n}$, $k=0,1,\dots, n$.
More specifically, we have the following

\begin{corollary}\label{cor:membership}
Let $v=\sum_{\alpha}f_\alpha\wv{\alpha}\in\extP{n}$ with $\a\in \IZ_2^n$ and 
$f_\alpha\in \Pol{n}$. Then $v\in\extL{n}$ if and only if, for any $k=0,\dots, n-1$, the 
elements $F_k:=f_{\tau^{(k)}}$ satisfy
\begin{enumerate}[(i)]
\item $F_k\in\Pol{n}^{\Sgp_{n-k}\times S_k}$;
\item for every $\alpha\in(\IZ_2^n)_k$, $f_{\alpha}=\dd{\sigma_{\alpha}}(F_k)$.
\end{enumerate}
\end{corollary}
\begin{proof}
Note that if $\a=\tau^{(k)}$, then $D_\a=\{n-k\}$.  Thus, the necessity of conditions $(i)$ and $(ii)$ are easy consequences of condition (iii) of Proposition \ref{prop:characterization1}.

Now we show that $(i)$ and $(ii)$ are sufficient conditions for membership $v\in\extL{n}$.  Fix $k\in \{1,\ldots,n\}$, and suppose $F_k\in \Pol{n}$ is given and satisfies $(i)$.  Define $f_\a:=\partial_{\sigma_\a}(F_k)$ for all $\a\in (\IZ_2^n)_k$, and set $v:=\sum_{\a\in (\Z_2^n)_k} f_\a \omega_\a$.  We must show that $\partial_i (f_\a)=0$ whenever $i\not\in D_\a$.
Let $w$ be the longest element of $\Sgp_{n-k}\times\Sgp_k\subset \Sgp_n$.  By $(i)$, $F_k\in \Pol{n}^{\Sgp_{n-k}\times\Sgp_k}$ is symmetric in the first $n-k$ variables and the last $k$ variables.  It follows that $F_k=\partial_w(G_k)$ for some polynomial $G_k$.  This is a straightforward generalization of the fact that $\symL{n}=\im \partial_{w_0}\subset \Pol{n}$, and follows easily from properties of the nilHecke algebra.

From the definition of $D_\a$, it is clear that
\[
\ell(s_i \sigma_\a \partial_w) = \ell(\sigma_\a\partial_w) -1
\]
if and only if $i\not \in D_\a$.  Thus, if $i\not\in D_\a$, then
\[
\partial_i(f_\a)=\partial_i\partial_{\sigma_\a}\partial_w(G_k)=0
\]
by properties of the divided difference operators.  This completes the proof.
\end{proof}

\subsubsection{Proof of Theorem \ref{thm:ext-sym}}\label{sss:end-pf-ext-sym}

Corollary \ref{cor:membership} gives us a map of $\symL{n}$-modules $\Phi_k:\Pol{n}^{\Sgp_{n-k}\times\Sgp_k}\rightarrow \extL{n}$ defined by
\[
\Phi_k(F) := \sum_{\a\in (\Z_2^n)_k} \partial_{\sigma_\a}(F) \omega_\a.
\]
Clearly $\Phi_k$ is injective, since $F$ can be recovered as the coefficient of $\omega_{\tau^{(k)}}$ in $\Phi_k(F)$.  By  Corollary \ref{cor:membership}, $\Phi_k$ surjects onto the component of $\extL{n}$ consisting of elements which are degree $k $ in the exterior variables $\omega_i$.  Since the dimension of $\Pol{n}^{\Sgp_{n-k}\times\Sgp_k}$ over $\symL{n}=\Pol{n}^{\Sgp_n}$ is $\binom{n}{k}$, statement $(i)$ of Theorem \ref{thm:ext-sym} follows.

Now, let $\{p_{\alpha}\}_{\alpha\in\IZ_2^n}$ be a solution of \eqref{eq:palphaNormalization} and set
\[
\swv{\alpha}(p_\a)=\wv{\alpha}+\sum_{\beta\succ\alpha}\dd{\sigma_{\beta}}(p_{\alpha})\cdot\wv{\beta}
\]
By \ref{cor:membership}, the elements $\swv{\a}(p_\a)$ belong to $\extL{n}$ and they are linearly
independent, since they are triangular with respect to $\{\wv{\alpha}\}$. By a dimension argument this induces an isomorphism of $\symL{n}$--modules
\[
\extL{n}\simeq\bigoplus_{\alpha\in\IZ_2^n}\symL{n}\cdot\swv{\alpha}(p_\a)
\]

This proves Theorem \ref{thm:ext-sym} $(ii)$.

For $(iii)$, suppose we have chosen elements
\[
\swv{i}=\wv{i} + \sum_{j>i} f_j\wv{j}\in \extL{n}
\]
for $i\in \{1,\ldots,n\}$.   Since these elements are degree 1 in the exterior variables, we have
\[
\swv{i}\swv{j}=-\swv{j}\swv{i}
\]
for every $1\leq i,j\leq n$.  Given the triangularity of $\{\swv{i}\}$ with respect to $\{\wv{i}\}$, the resulting map of rings
\begin{equation}\label{eq:wedgeInclusion}
\Wedge{\swv{1},\dots, \swv{n}} \rightarrow \extL{n}
\end{equation}
is clearly injective.  Extending linearly in $\symL{n}$ gives an injective map of $\symL{n}$-algebras
\[
\Psi:\symL{n}\ten\Wedge{\swv{1},\dots, \swv{n}}\rightarrow \extL{n}.
\]
By a dimension argument, $\Psi$ is surjective, and we obtain $(iii)$.

Finally, extending (\ref{eq:wedgeInclusion}) by $\Pol{n}$--linearity gives a $\Pol{n}$--algebra homomorphism
\begin{equation}\label{eq:iso-1}
\Pol{n}\ten\Wedge{\swv{1},\dots, \swv{n}}\to\Pol{n}\ten\Wedge{\wv{1},\dots, \wv{n}}=\extP{n},
\end{equation}
which we claim is an isomorphism.  Namely, the homomorphism is induced by
the nilpotent matrix $A$ with
coefficients in $\Pol{n}$ such that
\[
\underline{\wv{}}^{\scs{\mathsf{s}}}=(I+A)\underline{\wv{}}
\iff
\underline{\wv{}}=\sum_{i=0}^{n-1}(-1)^iA^i\underline{\wv{}}^{\scs{\mathsf{s}}}
\]
This determines the inverse to \eqref{eq:iso-1}. Applying the classical identification
$\Pol{n}\simeq \H_n\ten\symL{n}$, we get a ring isomorphism
\[
\extP{n}\simeq \H_n\ten\symL{n}\ten\Wedge{\swv{1},\dots,\swv{n}}\simeq \H_n\ten\extL{n}.
\]
In particular, $\extP{n}$ is a free module of rank $n!$ over $\extL{n}$ and there is a canonical
isomorphism
\[
\End_{\extL{n}}(\extP{n})\simeq\mathsf{Mat}(n!, \extL{n})
\]
which completes the proof of Theorem \eqref{thm:ext-sym}.


\subsection{Structure of the extended nilHecke algebra}\label{ss:structure}
The above analysis of $\extL{n}$ also has consequences for $\extNH{n}$.  Recall that $c[j]=c[j,n]$ denotes the permutation $s_{j}\cdots s_{n-1}$.

\begin{proposition}\label{prop:structureOfNH}
Let $p_j\in \Pol{n}^{S_{n-1}\times S_1}$ be polynomials of degree $n-j$ such that $\partial_{c[j]}(p_j)=1$, and set $\swv{j}:=\sum_i \partial_{c[i]}(p_j)$ as in Theorem \ref{thm:ext-sym}.  Then there is an isomorphism of algebras
\[
\extNH{n}\cong \NH{n}\otimes_\Q \Wedge{\swv{1},\ldots,\swv{n}}.
\]
The induced action on $\extP{n}\cong \Pol{n}\otimes_\Q\Wedge{\swv{1},\ldots,\swv{n}}$ is the standard action of $\NH{n}$ on $\Pol{n}$, tensored with the exterior algebra.  Consequently,
\begin{equation}\label{eq:extendedMatrixAlg}
\extNH{n}\cong \End_{\extL{n}}(\extP{n})\cong {\operatorname{Mat}}((n)^!_{q^2}, \extL{n})
\end{equation}
where $\extL{n}$ acts on $\extP{n}$ by right multiplication.
\end{proposition}
Note that $\extL{n}$ is graded commutative, hence in order for left multiplication by $\wv{i}\in \NH{n}$ on $\extP{n}$ to honestly commute with the action of $\swv{j}\extL{n}$ (as opposed to commutativity up to sign), it is necessary to choose the right action of $\extL{n}$ on $\extP{n}$ in (\ref{eq:extendedMatrixAlg}).
\begin{proof}
By definition $\extNH{n}$ contains $\NH{n}$ and $\extP{n}$ as subalgebras.  Tensoring the inclusion maps gives us an algebra map
\[
\NH{n}\otimes_{\Pol{n}} \extP{n} \rightarrow \extNH{n}.
\]
By Theorem \ref{thm:ext-sym}, we know that $\extP{n}\cong \Pol{n}\otimes_\Q \Wedge{\swv{1},\ldots,\swv{n}}$, hence the above reduces to an algebra map
\[
\NH{n}\otimes_{\Q} \Wedge{\swv{1},\ldots,\swv{n}}\rightarrow \extNH{n}.
\]
As a $\NH{n}$-module, the right hand side is isomorphic to $\NH{n}\otimes_\Q \Wedge{\wv{1},\ldots,\wv{n}}$.  From the definitions, it is clear that the $\swv{i}$ are unitriangular with respect to the $\wv{i}$, hence the above algebra map is an isomorphism.  This proves the first statement. The statement regarding the action on $\extP{n}$ is easily verified.  Finally \eqref{eq:extendedMatrixAlg} follows by combining the standard fact that $\NH{n}\cong \End_{\symL{n}}(\Pol{n})$ together with Theorem \ref{thm:ext-sym}, which states that $\extP{n}$ is free of rank $[n]!$ over $\extL{n}$.
\end{proof}

As an immediate corollary we have the following analogue of the usual fact that $\symL{n}=Z(\NH{n})$.

\begin{corollary}\label{cor:ZNH}
$\extL{n}$ is isomorphic to the graded center of $\extNH{n}$ as graded algebras.\qed
\end{corollary}
Here, the graded center of a graded algebra $A=\bigoplus_i A_i$ is the subset $gZ(A)\subset A$ consisting of homogeneous elements $z\in A$ such that $za = (-1)^{\deg(a)\deg(z)}az$ for every homogeneous $a\in A$.


\subsection{Bases of $\extL{n}$}\label{ss:basis}

We now discuss some explicit examples of bases of $\extL{n}$.
We adopt the following criteria. From Theorem \ref{thm:ext-sym},
a basis of $\extL{n}$ is determined by any family of 
elements $\{p_j\}_{1\leq j\leq n}\subset\Pol{n}$ satisfying
\begin{equation}\label{eq:basis-sys}
\dd{c[j]}(p_j)=1\aand p_j\in\Pol{n}^{\Sgp_{n-1}\times\Sgp_1}
\end{equation}
This allows to construct a ring isomorphism
\[
\extL{n}\simeq\symL{n}\ten\Wedge{\swv{1},\dots, \swv{n}}
\]
where
\[
\swv{j}=\sum_{k\geq j}\dd{c[k]}(p_j)\wv{k}.
\]
Any such collection $\{\swv{j}\}_{1\leq j\leq n}$ will be referred to as an
exterior basis of $\extL{n}$.

\subsubsection{Schubert polynomials}\label{sss:schubert}

The first example we discuss involves the use of Schubert polynomials.
Recall that the Schubert polynomials $\Sp{w}\in\Pol{n}$, with $w\in\Sgp_n$,
are a collection of polynomials indexed by elements of $\Sgp_n$ and
characterized by the following conditions:
\begin{enumerate}[(i)]
\item $\Sp{\id}=1$;
\item for every $u\in\Sgp_n$
\[
\dd{u}\Sp{w}=
\left\{
\begin{array}{cl}
\Sp{wu^{-1}} & \mbox{if}\;\;l(wu^{-1})=l(w)-l(u)\\
0 & \mbox{otherwise.}
\end{array}
\right.
\]
\end{enumerate}
More explicitly, one can check that
\[
\Sp{w}=\dd{w^{-1}w_0}\lp x_{1}^{n-1}x_{2}^{n-2}\cdots x_{n-1}\rp.
\]

\subsubsection{Schubert polynomials and $\extL{n}$}\label{sss:schubert-2}
The above characterization implies immediately the following
\begin{proposition}\label{prop:sch-basis}
The elements $p_j=\Sp{c[j]}$, ${1\leq j\leq n}$, are a solution of \eqref{eq:basis-sys}.
In particular, the elements
\[
\schwv{j}=\wv{j}+\sum_{k>j}\Sp{c[j,k]}\wv{k}
\]
define an exterior basis of $\extL{n}$.
\end{proposition}
\begin{proof}
It is clear from the definitions that $\Sp{c[j]}$ satisfy \eqref{eq:basis-sys}.  The proposition follows by an application of Theorem \ref{thm:ext-sym}.
\end{proof}

It is interesting to observe that the Schubert polynomials allow to define a solution
to the full system \eqref{eq:palphaNormalization}. Specifically, for every $\alpha\in(\IZ_2^n)_k$,
we can set $p_{\alpha}=\Sp{\sigma_{\alpha}}$. Then, it is easy to see that
$\dd{\sigma_{\beta}}\Sp{\sigma_{\alpha}}=\Sp{\sigma_{\alpha}\sigma_{\beta}^{-1}}$,
\[
\dd{\sigma_{\alpha}}\Sp{\sigma_{\alpha}}=1\aand \dd{i}\Sp{\sigma_{\alpha}}=0
\]
for every $i\neq n-k$.  Therefore, we get extended symmetric polynomials
\[
\schwv{\alpha}=\wv{\alpha}+\sum_{\beta\succ\alpha}\Sp{\sigma_{\alpha}\sigma_{\beta}^{-1}}\wv{\beta}
\in\extL{n}.
\]
In fact, these are exactly the elements of the standard basis of $\Wedge{\schwv{1},\dots,\schwv{n}}$.

\begin{proposition}\label{prop:sch-basis-2}
The standard basis of $\Wedge{\schwv{1},\dots, \schwv{n}}$
has the following description. For any $\alpha\in(\IZ_2^n)_k$,
\[
\schwv{1}^{\alpha_1}\cdots\schwv{n}^{\alpha_n}=
\wv{\alpha}+\sum_{\beta\succ\alpha}\Sp{\sigma_{\alpha}\sigma_{\beta}^{-1}}\wv{\beta}
\]
\end{proposition}

The proof will be carried out in \ref{sss:det-ident}, \ref{sss:pf-lem:sch-sch}
and \ref{sss:sch-basis-2}.

\subsubsection{Determinantal identities}\label{sss:det-ident}
In what follows we will make use of the following result, relating Schubert polynomials
of Grassmannian permutations to Schur functions.

\begin{proposition}[Proposition 2.6.8\cite{Man}]\label{prop:sch-sch}
If $w\in\Sgp_n$ is a Grassmannian permutation, and if $r$ is its unique descent, then
\[
\Sp{w} = \schu{\lambda(w)}(x_1,x_2, \cdots, x_r),
\]
where $\schu{\lambda(w)}$ is the Schur function in the variables $\{ x_1, \dots, x_r\}$
corresponding to the partition $\lambda(w)$.
\end{proposition}

\begin{example}\label{example:lehmer} \hfill
\begin{enumerate}
  \item If $w = c{[j]} = s_j s_{j+1} \dots s_{n-1} \in \Sgp_n$, then $w$ has a unique descent at position $n-1$.  The Lehmer code is $(0,\dots,0,1, \dots, 1,0)$ and the corresponding partition $\lambda(w)=(1^{n-j})$.  Hence,
$\Sp{c[j]} = \ef_{n-j}(\xv{1},\dots,\xv{n-1})$.

\item More generally, if $j<k$ and $w = c{[j,k]} =s_js_{j+1}\dots s_{k-1} \in \Sgp_n$, then $w$ is a Grassmannian permutation with a unique descent in position $k-1$.  The corresponding partition $\lambda(w)=(1^{k-j})$ and $\Sp{c[j,k]}=\ef_{k-j}(\xv{1},\dots,\xv{k-1})$.

\item The permutations $c^{(k)}=c{[k,n]}\cdots c{[2,n-k+2]}\cdot c{[1, n-k+1]}$ have a unique descent at position $n-k$.  The Lehmer code for $c^{(k)}$ has $\Lw_1 = \Lw_2 = \dots = \Lw_{n-k} = k$ and $\Lw_j =0$ for $j>n-k$.  It follows that $\Sp{c^{(k)}} = \schu{(k^{n-k})}(x_1,x_2, \dots, x_{n-k})$.

\item Generalizing all of the previous examples, $w=c{[k,n]}\cdot c{[k-1,n-1]}\cdots  c{[k-j+1, n-j+1]}$ has a unique descent at $n-j$, and $\Sp{w} = \schu{(j^{n-k})}(x_1,x_2, \dots, x_{n-j})$.
\end{enumerate}
\end{example}

Recall that Schur functions satisfy the second Jacobi--Trudi identity:
for every partition $\lambda$ of length $l(\lambda)$
\begin{equation}\label{eq:JT}
\schu{\lambda}=\det(\ef_{\lambda_i'+j-i})_{i,j=1}^{l(\lambda')}=\det(\ef_{\lambda_j'+i-j})_{i,j=1}^{l(\lambda')}
\end{equation}
where $\lambda'$ is conjugate to $\lambda$.
The proof of Proposition \ref{prop:sch-basis-2} relies on the following
\begin{lemma}\label{lem:sch-sch}
For any $\alpha\in(\IZ_2^n)_k$, let $\sfu{}=(\sfu{1},\dots,\sfu{k})$ and $\lambda_{\alpha}$ be,
respectively, the corresponding sequence of indices and the partition defined in \ref{sss:fund-perm}. Then
\[
\schu{\lambda_{\alpha}}=\det(\ef_{(n-k+i)-\sfu{j}})_{i,j=1}^k
\]
Moreover,
\begin{align*}
\schu{\lambda_{\alpha}}(x_1,\dots, x_{n-k})&=
\det(e_{n-k+i-\sfu{j}}(x_1,\dots, x_{n-k}))_{i,j=1}^k
=\det(e_{n-k+i-\sfu{j}}(x_1,\dots, x_{n-k+i-1}))_{i,j=1}^k
\end{align*}
\end{lemma}

\begin{example}
The result of Lemma \ref{lem:sch-sch} is addressing the following phenomenon.
Set $n=2$ and consider the permutation $s_2s_1$.
In this case we get
\[
\Sp{s_2s_1}=x_1^2=\det
\left[
\begin{array}{ll}
x_1&1\\
x_1x_2 & x_1+x_2
\end{array}
\right]
=
\det
\left[
\begin{array}{ll}
\ef_1(x_1)&1\\
\ef_2(x_1,x_2) & \ef_1(x_1,x_2)
\end{array}
\right]
\]
On the other hand, $s_2s_1$ has a unique descent at $1$, its partition is $[2]$, its conjugate is
$[1,1]$, and, by second Jacobi--Trudi identity,
\[
\schu{[2]}=
\det
\left[
\begin{array}{ll}
\ef_1&1\\
\ef_2 & \ef_1
\end{array}
\right]
= \ef_1^2-\ef_2
\]
These coincide when we input the set of variables $\{x_1\}$, namely
\[
\Sp{s_2s_1}=x_1^2=\ef_1^2(x_1)-\ef_2(x_1)=\schu{[2]}(x_1).
\]
\end{example}

\subsubsection{Proof of Lemma \ref{lem:sch-sch}}\label{sss:pf-lem:sch-sch}
The first statement is immediate. Namely, the second Jacobi--Trudi identity
for $\lambda_{\alpha}$ reads
\[
\schu{\lambda_{\alpha}}=\det(\ef_{\lambda_j'+i-j})_{i,j=1}^{k}=\det(\ef_{(n-k+i)-\sfu{j}})_{i,j=1}^k
\]
since
\[
\lambda'_j+i-j=n-k-\sfu{j}+j+i-j=(n-k+i)-\sfu{j}
\]
To prove the second statement, we proceed by induction on $k$.
For $k=1$ there is nothing to prove. For $k>1$, consider the
expansion of
\[
\sfD{}=\det(\ef_{n-k+i-\sfu{j}}(x_1,\dots, x_{n-k+i-1}))
\]
along the last row, i.e.
\[
\sfD{}=\sum_{j=1}^k\ef_{n-\sfu{j}}(x_1,\dots, x_{n-1})\cdot\sfM{j}
\]
where $\sfM{j}$ is the signed minor of the matrix obtained by removing the last
row and the $j$th column. By induction, $\sfM{j}$ depends exclusively on the
variables $x_1,\dots, x_{n-k}$, and
\[
\sfM{j}=(-1)^{k+j}\det(\ef_{n-k+i-\sfu{l}}(x_1,\dots, x_{n-k}))_{\substack{i=1,\dots, k-1\\ 
\;\;l=1,\dots,\hat{j},\dots, k}}
\]
Applying the usual recursive relation for elementary symmetric functions
\[
\ef_m(x_1,\dots, x_n)=\ef_{m}(x_1,\dots, x_{n-1})+x_n\ef_{m-1}(x_1,\dots, x_{n-1})
\]
we get
\begin{align*}
\sfD{}=&\sum_{j=1}^k\ef_{n-\sfu{j}}(x_1,\dots, x_{n-1})\cdot\sfM{j}=\\
=&\sum_{j=1}^k\ef_{n-\sfu{j}}(x_1,\dots, x_{n-2})\cdot\sfM{j}+
x_{n-1}\sum_{j=1}^k\ef_{n-1-\sfu{j}}(x_1,\dots, x_{n-2})\cdot\sfM{j}
\end{align*}
Now we observe that
\[
\sum_{j=1}^k\ef_{n-1-\sfu{j}}(x_1,\dots, x_{n-2})\cdot\sfM{j}=0
\]
since it describes the determinant of a matrix with two equal rows.
By iterating this process we get
\[
\sfD{}=\sum_{j=1}^k\ef_{n-\sfu{j}}(x_1,\dots, x_{n-k})\cdot\sfM{j}=
\det(\ef_{n-k+i-\sfu{j}}(x_1,\dots, x_{n-k}))_{i,j=1}^k
\]

\subsubsection{Proof of Proposition \ref{prop:sch-basis-2}}\label{sss:sch-basis-2}
\Omit{We have
\[
\schwv{\alpha}=\sum_{\beta\succeq\alpha}\Sp{\sigma_{\alpha}\sigma_{\beta}^{-1}}\wv{\beta}
\]
and $\schwv{j}=\schwv{\delta_j}$, $\delta_j=(0\dots 0 \stackrel{(j)}{1} 0 \dots 0)$.
}
Let $\Smx\in\mathsf{Mat}(n\times n, \Pol{n})$ be the unipotent lower triangular
matrix
\[
[\Smx]_{ij}=\Sp{s_j\cdots s_{i-j}}=\Sp{c[j,i]}=\ef_{i-j}(x_1,\dots, x_{i-1})
\]
for any $i>j$. The elements $\schwv{}=(\schwv{1}, \dots, \schwv{n})$ satisfy
$\schwv{}=\Smx^{\mathsf{T}}\wv{}$, where $\wv{}=(\wv{1},\dots, \wv{n})$.
In particular, their wedge product can be written in terms of minors of $\Smx$.
More specifically, for every $\alpha\in(\IZ_2^n)_k$,
\[
\schwvt{\alpha}:=\schwv{1}^{\alpha_1}\cdots\schwv{n}^{\alpha_n}=
\sum_{\beta\succeq\alpha}\sfD{\beta\alpha}\wv{\beta}
\]
where $\sfD{\beta\alpha}$ is the minor of $\Smx$ corresponding to the rows identified by $\beta$
and the columns identified by $\alpha$.

\begin{proposition}
For every $\alpha,\beta\in(\IZ_2^n)_k$, $\beta\succeq\alpha$,
$\Sp{\sigma_{\alpha}\sigma_{\beta}^{-1}}=\sfD{\beta\alpha}$.
In particular, $\schwv{\alpha}=\schwvt{\alpha}$.
\end{proposition}

\begin{proof}
Since $\sigma_{\alpha}$ is a Grassmannian permutation with descent at $n-k$, it follows
from Lemma \ref{lem:sch-sch}
\[
\Sp{\sigma_{\alpha}}=\schu{\lambda_{\alpha}}(x_1,\dots, x_{n-k})=
\det(\ef_{(n-k+i)-\sfu{j}}(x_1,\dots, x_{n-k}))_{i,j=1}^k
\]
and
\[
\Sp{\sigma_{\alpha}}=\det(\ef_{(n-k+i)-\sfu{j}}(x_1,\dots, x_{n-k}))_{i,j=1}^k=
\det(\ef_{(n-k+i)-\sfu{j}}(x_1,\dots, x_{n-k+i-1}))_{i,j=1}^k=\sfD{\tau^{(k)}\alpha}
\]
Moreover, since the elements $\schwv{j}$ are $\Sgp_n$--invariant, so is $\schwvt{\alpha}$. Hence
the coefficients $\sfD{\beta\alpha}$ satisfy
\[
\sfD{\beta\alpha}=\dd{\sigma_{\beta}}\sfD{\tau^{(k)}\alpha}
\]
and therefore
\[
\sfD{\beta\alpha}=\dd{\sigma_{\beta}}\sfD{\tau^{(k)}\alpha}=\dd{\sigma_{\beta}}\Sp{\sigma_{\alpha}}
=\Sp{\sigma_{\alpha}\sigma_{\beta}^{-1}}
\]
\end{proof}

This concludes the proof of Proposition \ref{prop:sch-basis-2}.

\begin{remark}\label{rem:reindex-sch}
It follows from the discussion above that the Schubert exterior basis of $\extL{n}$
is more concisely described in terms of elementary functions. It will be convenient
to reindex these elements. Henceforth, we will adopt the following notation
\[
\ewv{j}=\sum_{k=0}^{j-1}\ef_{k}(x_1,\dots, x_{n-j+k})\wv{n+1-j+k}=\schwv{n-j+1}
\]
\end{remark}
\subsubsection{Dual Schubert polynomials}\label{sss:dual-schubert}
Our second example of a basis for $\extL{n}$ relies on the notion
of \emph{dual} Schubert polynomials.

\begin{proposition}[\cite{Man} Proposition 2.5.7]
There is a $\symL{n}$-bilinear form on $\Pol{n}$ defined by $(x,y) := \partial_{w_0}(xy)$.  With respect to this form the dual basis to the Schubert polynomials are given by
\begin{equation}
  \dSp{w} = (-1)^{\ell(ww_0)}w_0
  (\Sp{ww_0}), \qquad w \in \Sgp_n.
\end{equation}
\end{proposition}

The dual Schubert polynomials are characterized by the following conditions:
\begin{enumerate}[$(i)$]
\item $\dSp{w_0}=1$;
\item for every $u\in\Sgp_n$
\[
\dd{u}\dSp{w}=
\left\{
\begin{array}{cl}
\dSp{wu^{-1}} & \mbox{if}\;\;l(wu^{-1})=l(w)+l(u)\\
0 & \mbox{otherwise.}
\end{array}
\right.
\]
\end{enumerate}
This follows directly from the characterization of the Schubert polynomials in \ref{sss:schubert}
and from the relation
\[
w_o\cdot\dd{u}\cdot w_0=(-1)^{l(u)}\dd{w_0uw_0}
\]
In particular, we get the following result, dualizing Proposition \ref{prop:sch-basis}.

\begin{proposition}
The elements $p_j=\dSp{w_0c[j]}$, ${1\leq j\leq n}$, are
a solution of \eqref{eq:basis-sys}. In particular, the elements
\[
\dschwv{j}=\wv{j}+\sum_{k>j}\dSp{w_0c[j,k]}\wv{k}
\]
define an exterior basis of $\extL{n}$.\qed
\end{proposition}
\Omit{
\MH{Ok to omit the proof?  It's a lot more obvious, given the simplified form of \eqref{eq:basis-sys}.}
\begin{proof}
One sees immediately that
$\dd{c[k]}\dSp{w_0c[j]}=\dSp{w_0c[j,k]}$, and therefore
\[
\dd{c[j]}\dSp{w_0c[j]}=1\aand \dd{i}\dSp{w_0c[j]}=0
\]
for every $i\neq n-1$.  The result follows.
\end{proof}
}
In \ref{example:lehmer}, we showed that the Schubert polynomials involved in
the exterior basis of $\extL{n}$ are elementary symmetric functions, namely,
\[
\Sp{c[j,k]}=\ef_{k-j}(x_1,\dots, x_{k-1}).
\]
The dual Schubert polynomials are, instead, naturally described by complete symmetric functions.
By definition, we have
\[
\dSp{w_0c[j,k]}=(-1)^{k-j}w_0(\Sp{w_0\cdot c[j,k]\cdot w_0})=(-1)^{k-j}w_0(\Sp{c[n-k+1,n-j+1]^{-1}})
\]
since $w_0\cdot c[j,k]\cdot w_0=s_{n-j}\cdots s_{n-k+1}$.
The permutation $c[n-k+1, n-j+1]^{-1}$ is still a Grassmannian permutation, whose unique descent
is at $n-k+1$ and whose partition is conjugate to that of $c[j,k]$. Therefore
\[
\Sp{c[n-k+1,n-j+1]^{-1}}=\hf_{k-j}(x_1,\dots, x_{n-k+1})
\]
and
\[
\dSp{w_0c[j,k]}=(-1)^{k-j}\hf_{k-j}(x_k,\dots, x_n).
\]
In particular, the relation $\dd{c[k]}\dSp{w_0c[j]}=\dSp{w_0c[j,k]}$ reads
\[
\dd{c[k]}((-1)^{n-j}\hf_{n-j}(x_n))=(-1)^{k-j}\hf_{k-j}(x_k, \dots, x_n)
\]
providing a different proof of \cite[Prop. 5.4]{abel-hogancamp}.

As in the Schubert case, one observes that the dual Schubert polynomials
give a solution of \eqref{eq:palphaNormalization}. Namely, one can
set $p_{\alpha}=\dSp{w_0\sigma_{\alpha}}$. Then
$\dd{\sigma_{\beta}}\dSp{w_0\sigma_{\alpha}}=\dSp{w_0\sigma_{\alpha}\sigma_{\beta}^{-1}}$,
\[
\dd{\sigma_{\alpha}}\dSp{w_0\sigma_{\alpha}}=1
\aand
\dd{i}\dSp{w_0\sigma_{\alpha}\sigma_{\beta}^{-1}}=0
\]
for every $\beta\succ\alpha$ and $i\not\in D_{\beta}$. It follows that there are elements in $\extL{n}$
\[
\dschwv{\alpha}=\wv{\alpha}+\sum_{\beta\succ\alpha}\dSp{w_0\sigma_{\alpha}\sigma_{\beta}^{-1}}\wv{\beta}
\]
which satisfy, in analogy with \ref{prop:sch-basis-2}, $\dschwv{\alpha}=(\dschwv{1})^{\alpha_1}\cdots(\dschwv{n})^{\alpha_n}$.

\begin{remark}\label{rem:reindex-dual-sch}
It follows from the discussion above that the dual Schubert exterior basis of $\extL{n}$
is concisely described in terms of complete functions. As before, it will be convenient
to reindex these elements. Henceforth, we will adopt the following notation
\[
\hwv{j}=\sum_{k=0}^{j-1}(-1)^k\hf_{k}(x_{n+1-j+k},\dots, x_{n})\wv{n+1-j+k}=\dschwv{n-j+1}
\]
\end{remark}

\Omit{
\subsubsection{Extended complete symmetric polynomials}\label{sss:complete}
Our second example involves complete symmetric polynomials.
For every $j\in\IZ_{\geq 0}$, the $j$th complete symmetric polynomial in the variables
$(x_1,\dots, x_n)$ is
\[
\hf_j(x_1,\dots, x_n)=\sum_{1\leq i_1\leq \cdots \leq i_j\leq n} x_{i_1}\cdots x_{i_j}
\]
\begin{proposition}
The elements $p_j=(-1)^{n-j}\hf_{n-j}(x_n)$ form a solution of \eqref{eq:basis-sys}.
In particular, the elements
\[
\hwv{j}=\wv{j}+\sum_{k>j}(-1)^{k-j}\hf_{k-j}(x_k,\dots,x_n)\wv{j}
\]
define an exterior basis of $\extL{n}$.
\end{proposition}
\begin{proof}
It follows from \cite[Prop. 5.4]{abel-hogancamp}
\[
\dd{c[k]}((-1)^{n-j}\hf_{n-j}(x_n))=(-1)^{k-j}\hf_{k-j}(x_k, \dots, x_n)
\]
for every $k>j$.
In particular, if we set $p_j=(-1)^{n-j}\hf_{n-j}(x_n)$, we get
\[
\dd{c[j]}(p_j)=\dd{c[j]}\lp(-1)^{n-j}\hf_{n-j}(x_n)\rp=\hf_0(x_j,\dots, x_n)=1
\]
and
\[
\dd{i}\dd{c[k]}(p_j)=\dd{i}\dd{c[k]}\lp(-1)^{n-j}\hf_{n-j}(x_n)\rp=\dd{i}\lp(-1)^{k-j}\hf_{k-j}(x_k, \dots, x_n)\rp=0
\]
for every $k>j$ and $i\neq k-1$. The result follows.
\end{proof}
}

\subsubsection{A family of bases for $\extL{n}$}\label{sss:interpolate}
We now describe a collection of bases of $\extL{n}$ which interpolates between
the Schubert basis \ref{sss:schubert} (described in terms of elementary symmetric functions)
and the dual Schubert basis \ref{sss:dual-schubert} (described in terms of complete symmetric functions).
\Omit{
For every $j\in\IZ_{\geq 0}$, the $j$th elementary symmetric polynomial in the variables $(x_1,\dots, x_n)$ is
\[
\ef_j(x_1,\dots, x_n)=\sum_{1\leq i_1< \cdots < i_j\leq n} x_{i_1}\cdots x_{i_j}
\]
}
Recall that the elementary symmetric functions satisfy the relation
\[
\ef_j(x_1,\dots, x_{n-1})=\sum_{l=0}^{j}(-1)^{j-l}x_n^{j-l}\ef_l(x_1,\dots, x_n)
\]
For every $0\leq r\leq n-j$, set
\begin{align*}
p_j^{(r)}=&(-1)^r\xv{n}^r\cdot \ef_{n-j-r}(\xv{1},\dots, \xv{n-1})
\\
=&(-1)^r\xv{n}^r\sum_{l=0}^{n-j-r}(-1)^{n-j-r-l}x_n^{n-j-r-l}\ef_l(x_1,\dots, x_n)
\\
=&\sum_{l=0}^{n-j-r}(-1)^{n-j-l}\hf_{n-j-l}(x_n)\ef_l(x_1,\dots, x_n)
\end{align*}
\begin{proposition}
For every choice of $r$, the elements $p_j^{(r)}$, $1\leq j\leq n$, are
a solution of \eqref{eq:basis-sys}. In particular, the elements
\[
\schwvr{j}=\sum_{k\geq j}\dd{c[k]}(p_j^{(r)})\wv{k}
\]
define an exterior basis of $\extL{n}$.
\end{proposition}
\begin{proof}
Since $\dd{c[k]}\lp \hf_{n-j-l}(x_n)\rp=0$ for every $k\geq j$ and $l>k-j$, we have
\[
\dd{c[k]}(p_j^{(r)})=\sum_{l=0}^{k-j}(-1)^{k-j-l}\ef_l(x_1,\dots, x_n)\cdot \hf_{k-j-l}(\xv{k},\dots, \xv{n})
\]
Therefore,
\[
\dd{c[j]}(p_j^{(r)})=1\aand\dd{i}\dd{c[k]}(p_j^{(r)})=0
\]
for every $k>j$ and $i\neq k-1$. The result follows.
\end{proof}
This basis interpolates between \ref{sss:schubert} and \ref{sss:dual-schubert}.
Specifically, for $r=0$,
\[
p_j^{(0)}=\ef_{n-j}(\xv{1},\dots,\xv{n-1})=\Sp{c[j]}
\]
and we obtain the Schubert exterior basis \ref{sss:schubert}.
Instead, for $r=n-j$,
\[p_j^{(n-j)}=(-1)^{n-j}\hf_{n-j}(\xv{n})=\dSp{w_0c[j]}\]
and we obtain the dual Schubert exterior basis \ref{sss:dual-schubert}.
Indeed, more precisely, we have, for any $0\leq r\leq n-j$,
\begin{equation}
p_j^{(r)}=\Sp{c[j+r]}\cdot\dSp{w_0c[n-r]}
\end{equation}

\begin{example}
Set $n=3$, then we have
\[
\begin{array}{|c|c|c|}
\hline
w & \Sp{w} & \dSp{w_0w}\\
\hline \hline
\id & 1 & 1\\
\hline
s_1 & x_1& -x_2-x_3 \\
\hline
s_2 & x_1+x_2 & -x_3\\
\hline
s_1s_2 & x_1x_2 & x_3^2 \\
\hline
s_2s_1 & x_1^2 & x_2x_3 \\
\hline
s_1s_2s_1 & x_1^2x_2 & x_2x_3^2 \\
\hline
\end{array}
\]
In particular, the Schubert exterior basis of $\extL{3}$ is
\begin{align*}
\schwv{1}=\wv{1}+x_1\wv{2}+x_1x_2\wv{3}
\qquad
\schwv{2}=\wv{2}+(x_1+x_2)\wv{3}
\qquad
\schwv{3}&=\wv{3}
\end{align*}
Instead, the dual Schubert exterior basis is
\begin{align*}
\dschwv{1}=\wv{1}-(x_2+x_3)\wv{2}+x_3^2\wv{3}
\qquad
\dschwv{2}=\wv{2}-x_3\wv{3}
\qquad
\dschwv{3}=\wv{3}
\end{align*}
Other possible choices are obtained replacing $\schwv{1}$ or $\dschwv{1}$ with
\[
\schwv{1}^{(1)}=\wv{1}+x_1\wv{2}-(x_1+x_2)x_3\wv{3}
\]
corresponding to the choice $p_1^{(1)}$ in \ref{sss:interpolate}.

\Omit{
$\extL{3} = \symL{3}\otimes \Lambda(\wv{1}^s,\wv{2}^s,\wv{3}^s)$ where
\begin{equation*}
\wv{i}^s = \wv{i} + \sum_{i<k<n}\mathcal{S}_{s_i s_{i+1} \cdots s_{k-1}} \wv{k}
\end{equation*}
and $\mathcal{S}_{s_i s_{i+1} \cdots s_{k-1}}$ are Schubert polynomials. For $\Sgp_3$ the Schubert polynomials are as follows:\\

\begin{equation*}
\begin{array}{rrr}
\mathcal{S}_{id}  = 1, & \mathcal{S}_{s_1} = x_1, & \mathcal{S}_{s_2} = x_1 + x_2\\
\mathcal{S}_{s_1 s_2} = x_1 x_2, & \mathcal{S}_{s_2 s_1} = x_1^2, & \mathcal{S}_{s_1 s_2 s_1} = x_1^2 x_2
\end{array}
\end{equation*}
Then we have :
\begin{equation*}
\begin{aligned}
\wv{1}^s &= \wv{1} + \sum_{1<k\leq 3} \mathcal{S}_{s_1 \cdots s_{k-1}} \wv{k}
      = \wv{1} + \mathcal{S}_{s_1}\wv{2} + \mathcal{S}_{s_1 s_2}\wv{3}
      = \wv{1} + x_1 \wv{2}+ x_1 x_2 \wv{3}\\
\wv{2}^s &= \wv{2} + \sum_{2<k\leq 3} \mathcal{S}_{s_2 \cdots s_{k-1}} \wv{k}
      = \wv{2} + \mathcal{S}_{s_2}\wv{3}
      = \wv{2} + (x_1 + x_2)\wv{3}\\
\wv{3}^s &= \wv{3}
\end{aligned}
\end{equation*}

\noindent Also let us verify that $\wv{1}^s, \wv{2}^s$, and $\wv{3}^s$ are in $\mathsf{ker}\dd{1} \cap \mathsf{ker}\dd{2}$
\begin{equation*}
\begin{aligned}
\dd{1}(\wv{1}^s) &= \dd{1}(\wv{1} + x_1\wv{2} + x_1x_2\wv{3})\\
                 &= -\wv{2} + \dd{1}(x_1)\wv{2} + x_1\dd{1}(\wv{2})-(x_1-x_2)\dd{1}(x_1)\dd{1}(\wv{2}) + \dd{1}(x_1x_2)\wv{3} + x_1x_2\dd{1}(\wv{3}) -(x_1-x_2)\dd{1}(x_1x_2)\dd{1}(\wv{3})\\
                 &= 0
\end{aligned}
\end{equation*}

\begin{equation*}
\begin{aligned}
\dd{2}(\wv{1}^s) &= \dd{2}(\wv{1} + x_1\wv{2} + x_1x_2\wv{3})\\
                 &= \dd{2}(\wv{1}) + \dd{2}(x_1)\wv{2} + x_1\dd{2}(\wv{2}) -(x_2-x_3)\dd{2}(x_1)\dd{2}(\wv{2}) + \dd{2}(x_1x_2)\wv{3} + x_1x_3\dd{2}\wv{3} - (x_2-x_3)\dd{2}(x_1x_3)\dd{2}(\wv{3})\\
                 &= 0
\end{aligned}
\end{equation*}
\noindent Similarly it also follows that $\dd{1}(\wv{2}^s) = \dd{2}(\wv{2}^s) = 0$ and $\dd{1}(\wv{3}^s) = \dd{2}(\wv{3}^s) = 0$.
}
\end{example}

\subsubsection{Other bases}
We conclude this section with two more examples.
\begin{itemize}
\item[(i)] \emph{Power functions}.
One can consider power symmetric polynomials and set
\[
p_j=(-1)^{n-j}\pf_{n-j}(x_1,\dots, x_{n-1})
\]
On the other hand,
\[
p_j=(-1)^{n-j}\pf_{n-j}(x_1,\dots, x_{n-1})=(-1)^{n-j}\pf_{n-j}(x_1,\dots, x_n)+(-1)^{n-j}\hf_{n-j}(x_n)
\]
Therefore it simply gives back the description in terms of complete symmetric functions.\\
\item[(ii)] \emph{Symmetrizers}
The easiest example, although computationally most expensive, is obtained by full symmetrization of the
exterior variables $\wv{j}$, i.e. for every $1\leq j\leq n$, set
\[
\swv{j}=\frac{1}{n!}\sum_{\sigma\in\Sgp_n}\sigma(\wv{j})
\]
\Omit{This is very obvious, but yet "obscure", I keep it here temporarily, just for record. It seems plausible that
the coefficient of $\swv{j}$ along $\wv{n}$ is $\frac{1}{n-j+1}\sum_{r=0}^{n-j}p_j^{(r)}$, see below, which is clearly
another good solution of the system above. I checked it only for $n=2,3$ though.}
\end{itemize}

\subsection{Combinatorial identities}
The following results give the relationship between $\{\ef_{i}^w\}$ and $\{\hf_{i}^w\}$ (see Remarks~\ref{rem:reindex-sch} and \ref{rem:reindex-dual-sch}), where
\begin{align}\label{eq:reindex}
\hwv{j}&=\sum_{k=0}^{j-1}(-1)^k\hf_{k}(x_{n+1-j+k},\dots, x_{n})\wv{n+1-j+k}\\
\ewv{j}&=\sum_{k=0}^{j-1}\ef_{k}(x_1,\dots, x_{n-j+k})\wv{n+1-j+k}
\end{align}

We use the following identity between elementary symmetric polynomials and complete homogeneous symmetric polynomials to prove the next proposition:
\begin{lemma}\label{eqn:e-h_different_numberof_variables}
\begin{equation}
\ef_{k}(x_1,\cdots,x_{n-j+k}) =\sum_{t=0}^{k}(-1)^{k+t}\hf_{k-t}(x_{n-j+k+1},\cdots,x_n)\ef_{t}(x_1,\cdots,x_n).
\end{equation}
\end{lemma}

\begin{proof}
Using standard facts about elementary and complete symmetric functions we have
\begin{align*}
&\sum_{t=0}^{k}(-1)^{t}\hf_{k-t}(x_{n-j+k+1},\cdots,x_n)\ef_{t}(x_1,\cdots,x_n)\\
&=\sum_{t=0}^{k}(-1)^{t}\hf_{k-t}(x_{n-j+k+1},\cdots,x_n) \left( \sum_{a=0}^{t}\ef_{a}(x_1,\cdots,x_{n-j+k})\ef_{t-a}(x_{n-j+k+1},\cdots,x_n)\right)\\
&=\sum_{a=0}^{k}\sum_{t=a}^{k}(-1)^{t}\hf_{k-t}(x_{n-j+k+1},\cdots,x_n)\ef_{t-a}(x_{n-j+k+1},\cdots,x_n)\ef_{a}(x_1,\cdots,x_{n-j+k})\\
&=\sum_{a=0}^{k}\sum_{t'=0}^{k-a}(-1)^{k-t'}\hf_{t'}(x_{n-j+k+1},\cdots,x_n)\ef_{k-a-t'}(x_{n-j+k+1},\cdots,x_n)\ef_{a}(x_1,\cdots,x_{n-j+k})\\
&=(-1)^k\sum_{a=0}^{k}\ef_{a}(x_1,\cdots,x_{n-j+k})\left(\sum_{t'=0}^{k-a}(-1)^{t'}\hf_{t'}(x_{n-j+k+1},\cdots,x_n)\ef_{k-a-t'}(x_{n-j+k+1},\cdots,x_n)\right)\\
&=(-1)^k\sum_{a=0}^{k}\ef_{a}(x_1,\cdots,x_{n-j+k})\left(\delta_{a,k}\right)\\
&= (-1)^k\ef_{k}(x_1,\cdots,x_{n-j+k}).
\end{align*}
\end{proof}

\begin{proposition}\label{prop:e_and_h}
For any $1\leq j \leq n$, we have
\begin{align*}
\ef_{j}^w = \sum_{k=0}^{j-1}\ef_{k}\hf_{j-k}^w.
\end{align*}
\end{proposition}

\begin{proof} Using the definition of $\hf_{j}^w$ we have
\begin{align*}
\sum_{k=0}^{j-1}\ef_{k}\hf_{j-k}^w
          &= \ef_{0}(x_1,\cdots,x_n)\hf_{j}^w + \ef_{1}(x_1,\cdots,x_n)\hf_{j-1}^w +\cdots + \ef_{j-1}(x_1,\cdots,x_n)\hf_{1}^w\\
          &= \wv{n+1-j} + \sum_{k=1}^{j-1} \left(\sum_{l=0}^{k-1}(-1)^{k-l}\hf_{k-l}(x_{n-j+k+1},\cdots,x_n)\ef_{l}(x_1,\cdots,x_n)\right ) \wv{n-j+k+1}\\
          &=\wv{n+1-j} + \sum_{n+1-j<k\leq n} \ef_{k-(n+1-j)}(x_1,\cdots,x_{k-1})\wv{k}\\
          &= \ef_{j}^w
\end{align*}
where the third equality follows from Lemma~\ref{eqn:e-h_different_numberof_variables} and the last step comes from change of variables.
\end{proof}

\Omit{
\subsection{Extended symmetric functions}

There are natural surjective algebra homomorphisms $\phi_{n+1}\maps \extL{n+1} \to \extL{n}$ determined by sending $w_i \to w_{i-1}$ and $x_i \to x_{i-1}$.  Our reindexing conventions  $\hwv{j} := \dschwv{n+1-j}$ were chosen so that the variables $\hwv{j}$ are stable with respect to these surjections
 \[
 \phi_{n+1}(\hf^w_j) := \phi_{n+1}(\dschwv{n+2-j}) = \dschwv{n+1-j}  = \hf^w_j.
 \]
The ring of {\em extended symmetric functions} $ \extL{}$ is defined to be the inverse limit of this system of inclusions.

Our results from the previous section motivate the following definition.

\begin{definition}
The super Hopf algebra of extended symmetric functions $\extL{}=(\extL{})_0 \oplus (\extL{})_1$ has generators
\[
 \left\{ \hf_{a} \mid a \geq 0  \right\}\in (\extL{})_0 \qquad
  \left\{ \hf_{a}^w \mid a \geq 0  \right\}\in (\extL{})_1 \qquad
\]
with $\hf_0=1$.  The generators satisfy the relations
\begin{equation}
 \hf_a \hf_b = \hf_b \hf_a, \qquad
 \hf^w_a \hf_b = \hf_b \hf^w_a, \qquad
 \hf^w_a \hf^w_b = -\hf^w_b \hf^w _a, \quad (\hf^w_a)^2 = 0.
\end{equation}

The coproduct $\Delta\maps \extL{} \to \extL{} \otimes \extL{}$ is defined on generators by
\begin{align}
  \Delta(\hf_{n}) &= \sum_{k} \hf_{k} \otimes \hf_{n-k} \\
  \Delta(\hf^w_{n}) &= \sum_{k} \hf_{k} \otimes \hf^w_{n-k}
\end{align}
The counit $\varepsilon \maps  \extL{} \to \Bbbk$ is defined by
\begin{equation}
  \varepsilon(\hf_{n}) =   \varepsilon(\hf^w_{n}) = \delta_{n,0}.
\end{equation}

\end{definition}
}


\section{Solomon's Theorem}\label{s:Sol-thm}
\nc\bfx{\mathbf{x}}
\nc\bfdx{\mathbf{dx}}
\nc\bff{\mathbf{f}}
\nc\bfdf{\mathbf{df}}
\nc\bfe{\mathbf{e}}
\nc\bfde{\mathbf{de}}
\nc\bfp{\mathbf{p}}
\nc\bfq{\mathbf{q}}
\nc\bfw{\boldsymbol{\omega}}
\nc\bfdw{\mathbf{d}\boldsymbol{\omega}}
\nc\bfal{\boldsymbol{\alpha}}
\nc\sfP{\mathsf{P}}
\nc\sfJ{\mathsf{J}}
\nc\sfQ{\mathsf{Q}}
\nc\inv{^{-1}}

\subsection{Superpolynomials and superinvariants}

Fix an integer $n\geq 1$.  Let $\bfx$ denote a set of formal even variables $x_1,\ldots,x_n$, and let $\bfdx$ denote a set of formal odd variables $dx_1,\ldots,dx_n$.  Here ``odd'' means that these variables are assumed to anti-commute amongst themselves and square to zero.  Thus, $\IQ[\bfx,\bfdx]$ is short-hand for the \emph{superpolynomial ring}
\[
\IQ[\bfx,\bfdx] := \IQ[x_1,\ldots,x_n]\otimes_\IQ \bigwedge[dx_1,\ldots,dx_n].
\]
We make this ring bigraded by declaring that $\deg(x_i) = (1,0)$ and $\deg(dx_i) = (0,1)$.

The symmetric group $\Sgp_n$ acts on $\IQ[\bfx,\bfdx]$ by algebra automorphisms, defined by permuting indices: $w(x_i)=x_{w(i)}$ and $w(dx_i)=dx_{w(i)}$.  Note that this action preserves the bidegree.
\begin{theorem}[Solomon \cite{Solomon}]\label{thm:solomon}
For any family $\bff=\{f_1,\dots,f_n\}$ of algebraically independent generators of $\IQ[\bfx]^{\Sgp_n}$,
\[
\IQ[\bfx,\bfdx]^{\Sgp_n}=\IQ[\bff,\bfdf]
\]
\end{theorem}

In particular,
\[
\IQ[x_1,\ldots,x_n,dx_1,\ldots,dx_n]^{\Sgp_n} = \IQ[e_1,\ldots,e_n,de_1,\ldots,de_n],
\]
where $e_i=e_i(x_1,\ldots,x_n)$ is the $i$-th elementary symmetric polynomial, and $de_i \in \IQ[\bfx,\bfdx]$ is to be interpreted in the usual manner for functions:
\[
df :=\sum_{i=1}^n \frac{\partial f}{\partial x_i}dx_i  \ \ \ \ \ \ \ \ \ \ \ \ \ \ \forall f\in \IQ[\bfx].
\]
Note that $\deg(e_i) = (i,0)$ and $\deg(de_i) = (i-1,1)$.

\Omit{
\begin{remark}
We have $de_1 = \sum_i dx_i$ and, more generally,
\[
de_i(x_1,\ldots,x_n) = \sum_{j=1}^n e_{i-1}(x_1,\ldots,\hat{x_i},\ldots,x_n) dx_i.
\]
\end{remark}
}
\begin{remark}
The mapping $f\mapsto df$ extends to a degree $(-1,1)$ differential $\IQ[\bfx,\bfdx]\rightarrow \IQ[\bfx,\bfdx]$.  This is the usual exterior derivative on polynomial differential forms.
\end{remark}

\subsection{Action of the extended nilHecke algebra}
Taking a cue from higher representation theory, we would like to consider divided difference operators $\partial_i$ acting on superpolynomials.  Unlike in the case of ordinary polynomials, here it is necessary to introduce rational functions in the variables $x_1,\ldots,x_n$.  So, let $\alpha_i:=x_i-x_{i+1}$ for $i=1,\ldots,n-1$, let $\bfal\inv=\{\a_1\inv,\ldots,\a_{n-1}\inv\}$, and consider the algebra $\IQ[\bfx,\bfdx,\bfal\inv]$.  Note that this algebra is bigraded, with $\deg((x_i-x_{i+1})\inv) = (-1,0)$.

We have the divided difference operators $\partial_i : \IQ[\bfx,\bfdx,\bfal\inv]\rightarrow \IQ[\bfx,\bfdx,\bfal\inv]$ defined in the usual way
\[
\partial_i = \frac{1-s_i}{x_i-x_{i+1}}.
\]
It follows from Solomon's theorem that for any tuple $\bff=\{f_1,\dots, f_n\}$ of algebraically independent
generators of $\IQ[\bfx]^{\Sgp_n}$ the subalgebra
$\IQ[\bfx,\bfdf]\subset \IQ[\bfx,\bfdx,\bfal\inv]$
is closed under the action of the divided difference operators.

Consequently, $\IQ[\bfx,\bfdf]$ is a module over the extended nilHecke algebra.  We wish to compare this module with the polynomial representation of the extended nilHecke algebra considered earlier.  This representation can be described as follows.  Let $\bfw =\{ \omega_1,\ldots,\omega_n\}$ be a set of
formal odd variables, with bidegree
\[
\deg(\omega_i)= (n-i,1).
\]
The superpolynomial ring $\IQ[\bfx,\bfw]$ admits an $\Sgp_n$ action via $w(x_i)=x_{w(i)}$ for all $w\in \Sgp_n$, together with
\[
s_j(\omega_i) = \begin{cases}
\omega_i+(x_i-x_{i+1})\omega_{i+1} & \text{ if $j=i$,}\\
\omega_i & \text{ otherwise.}
\end{cases}
\]
Note that the $\Sgp_n$ action preserves the bidegree. The actions of $\IQ[\bfx]$ and $\IQ[\Sgp_n]$
determines uniquely that of $\extNH{n}$.

Note that the graded dimensions of $\IQ[\bfx,\bfw]$ and $\IQ[\bfx,\bfdf]$ coincide.  Thus, it is natural to hope for a bidegree preserving isomorphism of $\extNH{n}$-modules $\IQ[\bfx,\bfw]\cong \IQ[\bfx,\bfdf]$.   Note that equivariance with respect to the $\extNH{n}$ action is equivalent to linearity with respect to $\IQ[\bfx]$, together with equivariance with respect to $\Sgp_n$.

\subsection{Preliminary computations}
\nc\col[1]{\gamma_{#1}}
\nc\row[1]{\rho_{#1}}
\nc\Mat{\mathsf{Mat}}

We say that a tuple $\bfp=\{p_1,\dots, p_n\}\subset\IQ[\bfx]$ is \emph{admissible} if $p_j\in\IQ[\bfx]^{\Sgp_{n-1}\times\Sgp_1}$, $\deg(p_j)=n-j$,
and $\dd{c[j]}p_j\in\IQ^{\times}$ for any $j=1,\dots, n$, where $c[j]=s_j\cdot s_{j+1}\cdots s_{n-1}$ and $c[n]=\id$.
This implies, in particular, that the matrix $\sfP=[\dd{c[j]}p_i]_{1\leq i,j\leq n}\in\mathsf{Mat}(n,\IQ[\bfx])$
is upper triangular and invertible.

We introduce the following operators. For any ring $R$ and any $k=1,\dots, n-1$, let
$\col{k}$ be the linear operator on $\Mat(m\times n, R)$ defined by
\[
\col{k}(A)_{ij}=\delta_{j,k+1}A_{ik}
\]
and let $\row{k}$ be the linear operator on $\Mat(n\times m, R)$ defined by
\footnote{In other words, $\col{k}$ gives back the $k$th column of $A$ in $(k+1)$th position,
while $\row{k}$  gives back the $(k+1)$th row of $A$ in position $k$.}
\[
\row{k}(A)_{ij}=\delta_{i,k}A_{k+1,j}.
\]
The following lemma gives a characterization of admissible tuples in terms of
the corresponding matrices, obtained through divided difference operators.
\begin{lemma}\label{lem:adm}\hfill
\begin{enumerate}[(i)]
\item
If $\bfp=\{p_1,\dots, p_n\}$ is an admissible tuple, then $\sfP$ satisfies
$\dd{k}(\sfP)=\col{k}(\sfP)$
for any $k=1,\dots, n-1$.
\footnote{The action of the divided difference operator is defined entrywise.}
\item
For any invertible $\sfQ=[q_{ij}]\in\mathsf{Mat}(n,\IQ[\bfx])$ such that
$\dd{k}(\sfQ)=\col{k}(\sfQ)$ for $k=1,\dots, n-1$, and $\deg(q_{ij})=j-i$,
the tuple $\bfq=\{q_{1n},\dots, q_{nn}\}$ is admissible and $\sfQ_{ij}=\dd{c[j]}q_{in}$.
\end{enumerate}
\end{lemma}

\begin{proof}
$(i)$ follows immediately from the fact that $p_i\in\IQ[\bfx]^{\Sgp_{n-1}\times\Sgp_1}$ and therefore
\[
\dd{k}\dd{c[j]}p_i=\delta_{j,k+1}\dd{c[k]}p_i
\]
Let now $\sfQ$ be a solution of $\dd{k}(\sfQ)=\col{k}(\sfQ)$. Then, for any $k=1,\dots, n-2$,
$\dd{k}q_{in}=0$ and $q_{in}\in\IQ[\bfx]^{\Sgp_{n-1}\times\Sgp_1}$, $i=1,\dots, n$.
Moreover,
\[
q_{i,k}=\dd{k}q_{i,k+1}=\cdots=\dd{k}\dd{k+1}\cdots\dd{n-1}q_{in}=\dd{c[k]}q_{in}
\]
Finally, since $\deg(q_{ij})=j-i$ and $\sfQ$ is invertible, it follows that $\dd{c[j]}q_{jn}\in\IQ^{\times}$.
Therefore $\bfq=\{q_{1n},\dots, q_{nn}\}$ is admissible. This proves $(ii)$.
\end{proof}

We now consider the following situation. Let
$\Theta=\{\theta_1,\dots, \theta_n\}, \Xi=\{\xi_1,\dots,\xi_n\}$
be two sets of algebraically independent elements in $\IQ[\bfx,\bfdx]$
such that $\deg(\theta_i)=(n-i,1)=\deg(\xi_i)$, $i=1,\dots, n$, and
let $\sfP\in\Mat(n,\IQ[\bfx])$ be the invertible matrix defined by the relation
\begin{equation}\label{eq:sol-2}
\Xi=\sfP\Theta
\end{equation}
Note that, necessarily, $\deg(p_{ij})=j-i$.

\begin{lemma}\label{lem:triple}
Any two of these equations imply the third:
\begin{enumerate}[(a)]
\item $\dd{k}(P)=\col{k}(P)$;
\item $\dd{k}(\Xi)=0$;
\item $\dd{k}(\Theta)=-\row{k}(\Theta)$.
\end{enumerate}
\end{lemma}

\Omit{
\begin{lemma}\hfill
\begin{enumerate}
\item
If $\dd{k}(\sfP)=\col{k}(\sfP)$, then
\[
\dd{k}\Xi=0\iff\dd{k}\Theta=-\row{k}(\Theta).
\]
\item If $\dd{k}(\Xi)=0$ and $\dd{k}(\Theta)=-\row{k}(\Theta)$, then
\[
$\dd{k}(\sfP)=\col{k}(\sfP)$
\]
\end{enumerate}
\end{lemma}
}

\begin{proof}
We first show that, if $\dd{k}(\sfP)=\col{k}(\sfP)$, then
\begin{equation}\label{eq:sol-1}
\dd{k}(\Xi)=0\quad\iff\quad\dd{k}\Theta=-\row{k}(\Theta).
\end{equation}
One easily checks that, since
\[
s_k(\col{k}(\sfP))=s_k(\dd{k}(\sfP))=\dd{k}(\sfP)=\col{k}(\sfP),
\]
it follows $\col{k}(\sfP)\Theta=s_k(\sfP)\row{k}(\Theta)$.
Now, from \eqref{eq:sol-2}, one gets
\begin{align*}
\dd{k}(\Xi)=&\dd{k}(\sfP)\Theta+s_k(\sfP)\dd{k}(\Theta)=\\
=&\col{k}(\sfP)\Theta+s_k(\sfP)\dd{k}(\Theta)=\\
=&s_k(\sfP)\left(\row{k}(\Theta)+\dd{k}(\Theta)\right)
\end{align*}
Therefore, \eqref{eq:sol-2} follows from the invertibility of $\sfP$.

It remains to show that, if $\dd{k}(\Xi)=0$ and $\dd{k}(\Theta)=-\row{k}(\Theta)$, then
$\dd{k}(\sfP)=\col{k}(\sfP)$. From \eqref{eq:sol-1},
\[
0=\dd{k}(\sfP)\Theta+s_k(\sfP)\dd{k}(\Theta)=\dd{k}(\sfP)\Theta-s_k(\sfP)\row{k}(\Theta)
\]
Denote by $P_1,\dots, P_n$ the column vectors of $\sfP$.
Since the component of $\Theta=\{\theta_1,\dots, \theta_n\}$ are algebraically independent
over $\IQ[\bfx]$, the equation $\dd{k}(\sfP)\Theta=s_k(\sfP)\row{k}(\Theta)$ implies
\[
\dd{k}P_i=\delta_{i,k+1}s_k(P_k)
\]
and therefore $\dd{k}{\sfP}=\col{k}(\sfP)$.
\end{proof}

\subsection{$\extNH{n}$--equivariant isomorphisms}

Let $\bff=\{f_1,\dots, f_n\}$ be a set of algebraically independent generators of $\IQ[\bfx]^{\Sgp_n}$,
with $\deg(f_i)=n-i$, $\bfp=\{p_1,\dots, p_n\}\subset\IQ[\bfx]$ an admissible tuple
and set $\sfP=[\dd{c[j]}p_i]_{i,j=1,\dots, n}$.

\begin{proposition}
For any choice of $\bff$ and $\bfp$, there is a unique $\IQ[\bfx]$--linear algebra homomorphism
\[
\sfJ_{\bfp}^{\bff}:\IQ[\bfx,\bfw]\to\IQ[\bfx,\bfdx,\alpha\inv]
\]
defined by the relation $\bfdf=\sfP\cdot\sfJ_{\bfp}^{\bff}(\bfw)$.
Moreover, $\sfJ_{\bfp}^{\bff}$ is injective, $\extNH{n}$--equivariant, and degree preserving.
\end{proposition}

\begin{proof}
Since $\bfp$ is admissible, the matrix $\sfP$ is invertible and
the algebra homomorphism $\sfJ_{\bfp}^{\bff}$ is uniquely determined by
the condition $\bfdf=\sfP\sfJ_{\bfp}^{\bff}(\bfw)$ and linearity in $\IQ[\bfx]$.

The injectivity of $\sfJ_{\bfp}^{\bff}$ follows from the invertibility of $\sfP$ and the
algebraic independence of the elements $\bff=\{f_1,\dots, f_n\}$ and $\bfdf=\{df_1,\dots, df_n\}$.

The $\Sgp_n$--equivariance follows from $\IQ[\bfx]$--linearity and Lemmas \ref{lem:adm}, \ref{lem:triple}.
Namely, since $\bfp$ is admissible, it follows from Lemma~\ref{lem:adm} that $\dd{k}(\sfP)=\col{k}(\sfP)$.
Then, since $\bfdf=\sfP\sfJ_{\bfp}^{\bff}(\bfw)$ and $\dd{k}(\bfdf)=0$, it follows from
Lemma~\ref{lem:triple} that $\dd{k}(\sfJ_{\bfp}^{\bff}(\bfw))=-\row{k}(\sfJ_{\bfp}^{\bff}(\bfw))$,
which is equivalent to
\[
s_i(\sfJ_{\bfp}^{\bff}(\wv{j}))=\sfJ_{\bfp}^{\bff}(\wv{j})+\delta_{ij}(x_i-x_{i+1})\sfJ_{\bfp}^{\bff}(\wv{i+1})
\]
and implies the $\Sgp_n$--equivariance of $\sfJ_{\bfp}^{\bff}$. The $\extNH{n}$--equivariance follows.
Finally, the fact that $\sfJ_{\bfp}^{\bff}$ preserves the degree is
a straightforward check.
\end{proof}

The construction of the homomorphism $\sfJ_{\bfp}^{\bff}$ allows us to compare the
description of the $\Sgp_n$--invariants in $\IQ[\bfx,\bfw]$ from Theorem \ref{thm:ext-sym}
and that of the $\Sgp_n$--invariants in $\IQ[\bfx,\bfdx]$  from Solomon's Theorem.
We obtain the following

\begin{corollary}
The homomorphisms $\sfJ_{\bfp}^{\bff}$ restricts to a canonical identification of
$\Sgp_n$--invariants. More specifically, there is a commutative
diagram
\[
\xymatrix{
\IQ[\bfx,\bfw]\ar[d]_{\sfJ_{\bfp}^{\bff}} & \ar[l]_{\beta_{\sfP}} \IQ[\bfx,\bfw_{\bfp}] \ar@{=}[d] &
\ar@{_(->}[l] \ar@{=}[d] \IQ[\bfx,\bfw_{\bfp}]^{\Sgp_n} & \ar@{=}[d] \ar@{=}[l] \IQ[\bff,\bfw_{\bfp}]\\
\IQ[\bfx,\bfdx,\alpha\inv] & \IQ[\bfx,\bfdf] \ar@{_(->}[l] & \IQ[\bfx,\bfdf]^{\Sgp_n} \ar@{_(->}[l]& \ar@{=}[l] \IQ[\bff,\bfdf]\\
}
\]
where $\beta_{\sfP}$ denotes the change of $\IQ[\bfx]$--basis defined by $\bfw_{\bfp}=\sfP\bfw$
and the vertical arrows send $\bfw_{\bfp}$ to $\bfdf$.
\end{corollary}

\subsection{Example}
\Omit{
Set $\bfe=\{\ef_n,\dots, \ef_1\}$ and $\bfde=\{d\ef_n,\dots, d\ef_1\}$, where
\[
d\ef_i=\sum_{k=1}^n \ef_i(x_1,\dots,\hat{x}_k,\dots,x_n)dx_k
\]
Then, from Theorem \ref{thm:solomon}, $\IQ[\bfx,\bfdx]^{\Sgp_n}=\IQ[\bfe,\bfde]$.
}
Let $\mathbf{h}=\{p_1,\dots, p_n\}$ be the admissible tuple with $p_j=(-1)^{n-j}\hf_{n-j}(x_n)$, and
let $\mathsf{H}$ be the corresponding matrix. In particular,
\[
\mathsf{H}_{ij}=\dd{c[j]}p_i=(-1)^{j-i}\hf_{j-i}(x_j,\dots, x_n)
\]
It is easy to see that the homomorphism $\sfJ_{\mathbf{h}}^{\bff}$ is defined by
\[
\sfJ_{\mathbf{h}}^{\bff}(\bfw)=\sfQ\bfdf\qquad\mbox{where}\qquad \sfQ_{ij}=\ef_{j-i}(x_{i+1},\dots, x_n)
\]
Similarly, let $\mathbf{e}=\{p_1,\dots, p_n\}$ be the admissible tuple with $p_j=\ef_{n-j}(x_1,\dots, x_{n-1})$, and
let $\mathsf{E}$ be the corresponding matrix. In particular,
\[
\mathsf{E}_{ij}=\dd{c[j]}p_i=\ef_{j-i}(x_1,\dots, x_{j-1})
\]
and the homomorphism $\sfJ_{\mathbf{e}}^{\bff}$ is defined by
\footnote{
Both computations follow easily from the 
relation between the generating series
of elementary and complete functions. More specifically, for $j>i$, one has
\[
\left(\sum_{k\geq0}(-1)^kt^k\hf_k(x_j,\dots, x_n)\right)\left(\sum_{k\geq0}t^k\ef_k(x_{i+1},\dots, x_j,\dots, x_n)\right)=\prod_{l=i+1}^{j-1}(1+tx_l)
\]
In particular, comparing the coefficients of $t^{j-i}$, we get
\[
\sum_{k=i}^j(-1)^{j-k}\hf_{j-k}(x_j,\dots, x_n)\ef_{k-i}(x_{i+1},\dots,x_n)=0
\]
which implies that the entries of $\mathsf{H}^{-1}$ are the polynomials $\ef_{j-i}(x_{i+1},\dots, x_n)$.
Similarly for $\mathsf{E}$.
}
\[
\sfJ_{\mathbf{e}}^{\bff}(\bfw)=\tilde{\sfQ}\bfde\qquad\mbox{where}\qquad \tilde{\sfQ}_{ij}=(-1)^{j-i}\hf_{j-i}(x_{1},\dots, x_i)
\]


\section{Differentials}
In this section we show that the differential $d_N$ on $\extNH{n}$ defined in section~\ref{subsec:diff} restricts to the ring of extended symmetric functions $\extL{n}$.  We identify the resulting DG-algebra as the Koszul complex associated to a certain regular sequence of symmetric polynomials in $\symL{n}$, whose cohomology is isomorphic to the cohomology ring of a Grassmannian.  We also define new deformed differentials $d_N^{\Sigma}$ on $\extNH{n}$ in section~\ref{sec:deformed-diff}.  The deformed differentials also restrict to $\extL{n}$ and the resulting cohomology of $(\extL{n},d_N^{\Sigma})$ is related to $GL(N)$-equivariant cohomology of a Grassmannian.

\subsection{The standard differential}
Recall that $\extNH{n}$ admits a differential $d_N$ for each $N\geq n-1$, defined by
\[
d_N(\omega_i)=(-1)^i h_{N-i+1}(x_1,\ldots,x_i) \hskip.3in d_N(x_i)=0 \hskip.3in d_N(\partial_i)=0
\]
for all $i$, together with the Leibniz rule.  Consequently, $d_N$ is linear with respect to the subalgebra $\NH{n}\subset \ext{N}_n$.  The following states that $\extL{n}$ is a DG-subalgebra of $\extNH{n}$ in a natural way.

\begin{proposition}
The differential $d_N$ restricts to a differential on $\extL{n}\subset \extNH{n}$.
\end{proposition}
\begin{proof}
The subset $\extL{n}=Z(\extNH{n})\subset \extNH{n}$ can be characterized as the set consisting of those elements $z\in \extNH{n}$ such that $[\partial_i,z]=0$ for all divided difference operators $\partial_i\in \extNH{n}$.  On the other hand $d_N$ is $\NH{n}$-linear, so
\[
[\partial_i,d_N(z)]=d_N([\partial_i,z])=0
\]
if $[\partial_i,z]=0$.
\end{proof}

\Omit{
We will show that $(\extL{n},d_N)$ is the center of the DG-algebra $(\extNH{n},d_N)$ for all $N$.
 In particular, we show that $(\extL{n},d_N)$ is a DG-subalgebra of $(\extNH{n},d_N)$.  This amounts to proving that the differential $d_N$ preserves the ring of symmetric functions.  We will show that
\[
d_N \maps \extL{n} \to \Lambda_n,
\]
while sending traditional symmetric functions to zero.}

\begin{example}
Let us consider the differential $d_N$ of $\hwv{j}$.  We will see that $d_N(\hwv{j})$ lands in $\extL{n}$, by direct computation.
Recall from Remark~\eqref{rem:reindex-dual-sch} that for $n=3$
\begin{align*}
\hwv{1} &= \wv{3},\\
\hwv{2} &= \wv{2} -x_3\wv{3},\\
\hwv{3} &= \wv{1} - (x_2+x_3)\wv{2} = {x_3}^2\wv{3}.
\end{align*}
Then the differentials are computed as follows.
\begin{align*}
d_N(\hwv{1}) &= d_N(\wv{3}) = (-1)^3\hf_{N-2}(x_1,x_2,x_3)
 \\
d_N(\hwv{2}) &= d_N(\wv{2} -x_3\wv{3})
              = \hf_{N-1}(x_1,x_2) + x_3 \hf_{N-2}(x_1,x_2,x_3)\\
              &= \hf_{N-1}(x_1,x_2,x_3).
\end{align*}
The last equality comes from the following observation:
\begin{equation*}
\{(a,b,c)| a+b+c= N-1\} = \{(a,b,0)| a+b= N-1\} \cup \{(a,b,c)| a+b+c= N-1, c\geq 1 \}
\end{equation*}
Similarly
\begin{align*}
d_N(\hwv{3}) &= d_N(\wv{1} - (x_2+x_3)\wv{2} + x_3^2\wv{3}) \\
              &= -x_1^N - (x_2+x_3)\hf_{N-1}(x_1,x_2) - x_3^2 \hf_{N-2}(x_1,x_2,x_3)\\
              &= - \hf_{N}(x_1,x_2,x_3).
\end{align*}
Similar to the above argument, the last equality follows from the  observation:
\begin{align*}
\left\{ (a,b,c) \mid a+b+c = N  \right\}
 &=
\left\{ (a,0,0)  \mid a = N  \right\}
\cup    \left\{ (a,b,0)  \mid a + b= N, \; b \geq 1  \right\}   \\
& \qquad
\cup    \left\{(a,b,1) \mid a + b= N-1 \right\}
\cup    \left\{(a,b,c) \mid  a + b + c = N, \;  \text{and } c\geq 2 \right\}.
\end{align*}
\end{example}

Before we compute $d_N(\hwv{j})$ in general, we need the following result on symmetric functions.

\begin{lemma} \label{lem:sym-ident}
Let $\hf_{i}(x_j,\cdots, x_n)$ denote the complete homogeneous symmetric polynomial of degree $i$ in variables $x_j,\cdots, x_n$, for $1\leq j \leq n$. Then for any $1\leq i \leq n$ and $N\in\mathbb{N}$
\begin{equation}
\hf_{N-i+1}(x_1,\cdots,x_n) = \sum_{j=0}^{n-i} \hf_{N-i-j+1}(x_1,\cdots, x_{i+j})\hf_{j}(x_{i+j},\cdots, x_n).
\end{equation}
\end{lemma}

\begin{proof}
For any $0\leq j \leq n-i$, $1\leq i \leq n$, and $N\in\mathbb{N}$
\begin{align*}
& \hf_{N-i-j+1}(x_1,\cdots, x_{i+j})\hf_{j}(x_{i+j},\cdots, x_n) \hspace{1.5in}\\
& \quad =  \left(\sum_{b_1+\cdots +b_{i+j} = N-i-j+1} x_1^{b_1} \cdots x_{i+j}^{b_{i+j}}\right)\left(\sum_{a_{i+j}+\cdots +a_n = j} x_{i+j}^{a_{i+j}} \cdots x_{n}^{a_n}\right)\\
& \quad =\sum_{b_1+\cdots +b_{i+j} = N-i-j+1}  \sum_{a_{i+j}+\cdots +a_n = j} x_1^{b_1} \cdots x_{i+j}^{b_{i+j}+a_{i+j}} x_{i+j+1}^{a_{i+j+1}}\cdots x_n^{a_n} \\
& \quad =\sum_{k=0}^{j}\left(\sum_{b_1+\cdots +b_{i+j} = N-i-j+1}\left(\sum_{a_{i+j}+\cdots +a_n = j-k} x_1^{b_1} \cdots x_{i+j}^{b_{i+j}+k} x_{i+j+1}^{a_{i+j+1}}\cdots x_n^{a_n}\right)\right)
\end{align*}
The exponent of each monomial in above sum is an $n-$tuple $(b_1,\cdots,b_{i+j}+k,a_{i+j+1},\cdots,a_n)$
where
\begin{align*}
   b_1 +\cdots + b_{i+j} &= N-i-j+1,\\
a_{i+j+1} + \cdots + a_n &= j-k, \text{ and} \\
                 a_{i+j} &=k \text{ for any } 0\leq k \leq j.
\end{align*}
As $j$ varies in the range $0 \leq j \leq n-i$ these exponents exhaust uniquely all monomials appearing in $\hf_{N-i+1}(x_1,\cdots,x_n)$.

\end{proof}

\subsection{Koszul complex}
Let $R$ be a commutative ring, and let $a_1,\ldots,a_r\in R$ be given elements. The \emph{Koszul complex} associated to $(a_1,\ldots,a_r)$ is the DG algebra
\[
R\otimes \Wedge{\theta_1,\ldots,\theta_r}
\]
with $R$-linear differential uniquely characterized by $d(\theta_i)=a_i$ together with the graded Leibniz rule.  For the purposes of the Leibniz rule, the grading places $R$ in homological degree zero, and each $\theta_i$ in homological degree $-1$.

\begin{proposition} \label{prop:dN-restricts}
As a DG-algebra, $\extL{n}$ is isomorphic to the Koszul complex associated to $(-1)^i h_{N-i+1}\in \symL{n}$ ($1\leq i\leq n$).
\end{proposition}

\begin{proof}
\Omit{
By Theorem \ref{thm:ext-sym} we know that $\extL{n} \simeq\symL{n}\ten\Wedge{\wv{1}^s, \dots, \wv{n}^s}$ and $v \in \extL{n}$ if and only if $v = \sum f_{(i_1,\cdots, i_k)}\wv{i_1}^s\cdots \wv{i_k}^s$ then

\begin{equation*}
d_N(v) = \sum d_N(f_{(i_1,\cdots, i_k)}\wv{i_1}^s\cdots \wv{i_k}^s)
\end{equation*}
\hfill\\
Since $d_N$ is linear, it is enough to consider differential of $f \wv{i_1}^s\cdots \wv{i_k}^s$ for any $f\in\symL{n}$ and for any $k-$tuple $(i_1,\cdots, i_k)$ where $k\leq n$ and $1\leq i_1 < \cdots < i_k \leq n$.

\begin{equation*}
\begin{aligned}
d_N(f \wv{i_1}^s\cdots \wv{i_k}^s) &= d_N(f)\wv{i_1}^s\cdots \wv{i_k}^s + (-1)^{deg(f)} f d_N(\wv{i_1}^s\cdots \wv{i_k}^s)\\
&= f d_N(\wv{i_1}^s\cdots \wv{i_k}^s)\\
&= f\left(d_N(\wv{i_1}^s)\wv{i_2}^s\cdots \wv{i_k}^s + (-1)^{deg(\wv{i_1}^s)}\wv{i_1}^s d_N(\wv{i_2}^s\cdots \wv{i_k}^s) \right).
\end{aligned}
\end{equation*}

Therefore it is enough to check $d_N(\dschwv{i})$ where}

By Theorem \ref{thm:ext-sym} we know that $\extL{n} \simeq\symL{n}\ten\Wedge{\swv{1}, \dots, \swv{n}}$,
where $\swv{j}=\swv{j}(p_j)$ are determined by any choice of $p_j\in \Q[x_1,\ldots,x_n]^{S_{n-1}\times S_1}$
such that $\dd{j}\dd{j+1}\cdots\dd{n-1}(p_j)=1$.
For the purposes of computing the differential, it is especially convenient to work with the choice of $p_j$ as
constructed in \ref{sss:dual-schubert}.  In this case the resulting elements $\swv{i}$ are given by
\begin{equation*}
\dschwv{i}:= \sum _{j=0}^{n-i} (-1)^j\hf_j{(x_{i+j}, \cdots, x_n)}\wv{i+j}.
\end{equation*}
We know that the differential $d_N$ is linear with respect to the subalgebra $\symL{n}$ (this follows from $d_N(x_i)=0$ and the Leibniz rule), hence to prove the Proposition we need only show that $d_N(\dschwv{i})=(-1)^{i} \hf_{N-i+1}(x_1,\cdots, x_n)$.  Compute:
\begin{equation*}
\begin{aligned}
d_{N}(\dschwv{i})
                 =& \sum_{j=0}^{n-i} (-1)^j \hf_{j}(x_{i+j},\cdots, x_n)d_{N}(\wv{i+j}) \\
                 =& \sum_{j=0}^{n-i} (-1)^j \hf_{j}(x_{i+j},\cdots, x_n) (-1)^{i+j} \hf_{N-i-j+1}(x_1,\cdots, x_{i+j})\\
                 =& (-1)^{i} \hf_{N-i+1}(x_1,\cdots, x_n)
\end{aligned}
\end{equation*}
where the last equality follows from lemma~\ref{lem:sym-ident}. \Omit{Therefore, $d_N(\dschwv{i}) \in \extL{n}$ and $\left(\extL{n}, d_N\right)$ is a DG-subalgebra.}
\end{proof}

A sequence of elements $\mathbf{a}=(a_1,\ldots,a_r)\in R$ is called a \emph{regular sequence} if
\begin{itemize}
\item $a_1$ is not a zero divisor.
\item $a_i$ is not a zero divisor in $R/\la a_1,\ldots,a_{i-1}\ra$ for all $2\leq i\leq n$.
\end{itemize}
If $\mathbf{a}$ is regular, then the associated Koszul complex $K(\mathbf{a})$ has cohomology only in degree zero, where it is isomorphic to $R/\la a_1,\ldots,a_r\ra$.   Said differently, if $\mathbf{a}$ is a regular sequence then the canonical projection $K(\mathbf{a})\rightarrow R/\la a_1,\ldots,a_r\ra$ is a quasi-isomorphism.
\begin{corollary}
The DG-algebra $(\extL{n}, d_N)$ is quasi-isomorphic to the cohomology ring $H^{\ast}(Gr(n,N))$.
\end{corollary}

\begin{proof}
The sequence $h_N,h_{N-1},\ldots,h_{N-n+1}\in \symL{n}$ is a regular sequence, see for example~\cite[Proposition 7.2]{Wu}.  Thus, the cohomology of the associated Koszul complex is isomorphic to the quotient $\symL{n}/\la h_N,h_{N-1},\ldots,h_{N-n+1}\ra$, which is known to be isomorphic to $\Lambda_n/\la \hf_{N-n+1}\ra \cong H^{\ast}(Gr(n,N))$.
\end{proof}

\subsection{Deformed differentials} \label{sec:deformed-diff}

\subsubsection{Deformed cyclotomic quotients}
The cyclotomic quotients of the nilHecke algebra, and KLR algebras more generally, admit deformations called deformed cyclotomic quotients defined in~\cite{Web}. For us the most relevant reference is \cite[Section 3.2]{RoseW}. 

Let $\kappa_1,\ldots,\kappa_N\in \C$ be given, and let $\Sigma$ denote the root multiset consisting of pairwise distinct complex numbers $\lambda_1, \dots, \lambda_{\ell}$ corresponding to the roots of the polynomial
\begin{equation} \label{eq:root-multi}
P(x) = x^N + \sum_{j=1}^{N}\kappa_{j}x^{N-j},
\end{equation}
with multiplicities $N_1, \dots, N_{\ell}$.
For each $N>0$ define the {\em deformed cyclotomic ideal} $I^{\Sigma}_N$ associated to $\Sigma$ is the ideal of $\NH{n}$ defined by
\begin{equation}
I^{\Sigma}_N := \left\langle \sum_{j=0}^{N}\kappa_{j}x_1^{N-j} \right\rangle, \qquad \kappa_i \in \mathbb{C}
\end{equation}
where we take $\kappa_0=1$.  We define the {\em deformed cyclotomic quotient} $$\NH{n}^{\Sigma}:= \NH{n}/ I^{\Sigma}_N.$$

In \cite[Section 3.2]{RoseW} it is shown that the deformed cyclotomic quotient rings $\NH{n}^{\Sigma}$ are isomorphic to matrix rings of size $n!$ with coefficients in the $GL(N)$-equivariant cohomology ring $H^{\ast}_{GL(N)}(Gr(n,N))$ with equivariant parameters equal to $\und{\kappa}=(\kappa_1, \kappa_2, \dots, \kappa_N)$.  We denote this specialization by $H^{\Sigma}_{n}$.  If the parameters $\und{\kappa}$ are left generic, then the center of the deformed cyclotomic quotient is just the $GL(N)$-equivariant cohomology itself~\cite[Theorem 2.10]{Wu2}.

\begin{theorem}[Theorem 13 \cite{RoseW}] \label{thm:RW}
There is an algebra isomorphism
\[
 H^{\Sigma}_n \cong \bigoplus_{ \overset{\sum n_j =n}{0 \leq n_j \leq N}} \bigotimes_{j=1}^{\ell} 
 H^{\ast}(Gr(n_j,N_j)).
\]
\end{theorem}

We will realize both the deformed cyclotomic quotient $\NH{n}^{\Sigma}$ and the rings $H^{\Sigma}_n$ within the context of the extended nilHecke algebra.   For these realization we make use of the following lemma.

\begin{lemma} \label{lem:deformed-ideal} The following identities hold in $\NH{n}^{\Sigma}$
\begin{enumerate}
  \item For any $y\in\mathbb{N}$, $$\sum_{j=0}^{N}\kappa_{j}\left(x_{i+1}^{y+(N-j)}\dd{i}\right) = \sum_{j=0}^{N}\kappa_{j}\left(\dd{i}x_{i+1}^{y+(N-j)}\right)$$
  \item For any $y\in\mathbb{N}$, $$\sum_{j=0}^{N}\kappa_{j}\left(x_{i+1}^{y+(N-j)}\dd{i}\right)=0. $$

  \item For any $1 \leq i \leq n$, $$\sum_{j=0}^{N}\kappa_{j}x_{i}^{N-j}=0.$$

 \item For any $m\leq N$,
  $$\sum_{j=0}^{N-m+1}\kappa_{j}\sum_{\sum a_i =(N-m+1-j)} x_1^{a_1} x_2^{a_2}\dots x_{m}^{a_m} =\sum_{j=0}^{N-m+1}\kappa_j \hf_{N-m+1-j}(x_1,\cdots,x_m)=0.$$

\end{enumerate}

\end{lemma}

\begin{proof} The first claim follows from \eqref{eq:ind-dotslide}.  This implies
\begin{align*}
\sum_{j=0}^{N}\kappa_{j}\left(x_{i+1}^{y+(N-j)}\dd{i}\right) &= -\sum_{j=0}^{N}\kappa_{j}\left(x_{i+1}^{y+(N-j)}\dd{i}x_{i+1}\dd{i} \right) = -\sum_{j=0}^{N}\kappa_{j}\left( \dd{i}x_{i+1}^{y+(N-j)+1} \dd{i}\right)\\
&= -\sum_{j=0}^{N}\kappa_{j}\left( x_{i+1}^{y+(N-j)+1}\dd{i}^2 \right) = 0
\end{align*}
proving the second claim.   For the third identity use nilHecke relations \eqref{eq:nilHecke} and the second identity,
\begin{align*}
\sum_{j=0}^{N}\kappa_jx_{i+1}^{N-j} &= \sum_{j=0}^{N}\kappa_j\left( x_{i+1}^{N-j}\dd{i}x_{i} - x_{i+1}^{N-j+1}\dd{i}\right) = \sum_{j=0}^{N}\kappa_jx_{i+1}^{N-j}\dd{i}x_{i} - \sum_{j=0}^{N}\kappa_jx_{i+1}^{N-j+1}\dd{i} = 0.
\end{align*}

The last claim is proven by induction. Using nilHecke relations \eqref{eq:nilHecke} we have
\begin{equation*}
\sum_{j=0}^{N}\kappa_{j}\left(\dd{i}x_{i}^{N-j}\right) - \sum_{j=0}^{N}\kappa_{j}\left(x_{i}^{N-j}\dd{i}\right) = \sum_{j=0}^{N-1}\kappa_{j}\left(\sum_{a+b=y + (N-j)-1} x_{i}^a x_{i+1}^b\right).
\end{equation*}
The induction step is identical to Proposition 2.8 in \cite{Lau4}.
\end{proof}

\subsubsection{Deformed differentials}

Let $\Sigma$ denote the root multiset corresponding to the roots and multiplicities of the polynomial \eqref{eq:root-multi}.  To each   $\Sigma$ define a differential $d_N^{\Sigma}$ on $\NH{n}\ext$, which we call {\em deformed differential},  by
\begin{equation}
 \begin{tabular}{rrr}
$d_N^{\Sigma}(\dd{i})=0$,&   $d_N^{\Sigma}(x_i)=0$,&  and \quad $d_N^{\Sigma}(\wv{i}) = \sum_{j=0}^{N-i+1}(-1)^{i+1}\kappa_j \hf_{N-i+1-j}(x_1,\cdots,x_i)$.
\end{tabular}
\end{equation}

\begin{proposition} \label{prop:deformed differential satisfies relations} The map $d_N^{\Sigma}$ satisfies the relations
\begin{enumerate}
\item  $ \dd{i}d_N^{\Sigma}(\wv{i+1}) = d_N^{\Sigma}(\wv{i+1})\dd{i}$
\item $\dd{1}d_N^{\Sigma}(\wv{i}) + d_N^{\Sigma}(\wv{i+1})x_{i+1}\dd{i} = d_N^{\Sigma}(\wv{i})\dd{i} + \dd{i}x_{i+1}d_N^{\Sigma}(\wv{i+1})
$
\end{enumerate}
for all $1 \leq i \leq n$.
\end{proposition}

\begin{proof}
The first identity holds since $d_N^{\Sigma}(\wv{i+1})$ is symmetric in $x_i$ and $x_{i+1}$.
For  the second identity we compute
\begin{align*}
d_N^{\Sigma}(\wv{i+1})x_{i+1}\dd{i}
&=
-\sum_{j=0}^{N-i}(-1)^{i+1}\kappa_j \sum_{a+b=N-i+1-j}\hf_{a}(x_1,\cdots,x_{i-1},x_{i+1})\dd{i}\\
& \quad+
\sum_{j=0}^{N-i}(-1)^{i+2}\kappa_j \left(\sum_{a+b=N-i-j}\hf_{a}(x_1,\cdots,x_i)
  \sum_{k' + \ell' = b-1}x_{i}^{k'+1}x_{i+1}^{\ell' +1}\dd{i}\right)
\end{align*}

Similarly,
\begin{align*}
\dd{i}x_{i+1}d_N^{\Sigma}(\wv{i+1})
&=
-\dd{i}\sum_{j=0}^{N-i}(-1)^{i+1}\kappa_j\hf_{a}(x_1,\cdots,x_{i-1},x_{i+1})\\
&\quad+
\sum_{j=0}^{N-i}(-1)^{i+2}\kappa_j \dd{i} \left( \sum_{a+b=N-i-j} \hf_{a}(x_1,\cdots,x_{i-1})\sum_{k'+\ell'=b-1}x_{i}^{k'+1}x_{i+1}^{\ell' +1}\right)
\end{align*}
Therefore,
\begin{align*}
d_N^{\Sigma}(\wv{i+1})x_{i+1}\dd{i} - \dd{i}x_{i+1}d_N^{\Sigma}(\wv{i+1})
&=
-\sum_{j=0}^{N-i}(-1)^{i+1}\kappa_j \sum_{a+b=N-i+1-j}\hf_{a}(x_1,\cdots,x_{i-1},x_{i+1})\dd{i} \\
&\qquad +\dd{i}\sum_{j=0}^{N-i}(-1)^{i+1}\kappa_j \hf_{N-i+1-j}(x_1,\cdots,x_{i-1},x_{i+1})
\end{align*}
and the result follows using \eqref{eq:ind-dotslide}.
\end{proof}

\begin{corollary}
The deformed differential $d_N^{\Sigma}$ defines a degree $-1$ differential on $\extNH{n}$.
\end{corollary}

\begin{proof}
The only nontrivial relations to verify are proven in Proposition~\ref{prop:deformed differential satisfies relations}.
\end{proof}

\begin{theorem}
The $DG-$algebra $(\extNH{n},d_N^{\Sigma})$ is quasi-isomorphic to deformed cyclotomic quotient of the nilHecke algebra
$\NH{n}^{\Sigma} = \NH{n}/\la (\sum_{j=0}^{N}\kappa_{j}x_1^{N-j}) \ra$.
\end{theorem}

\begin{proof}
The statement follows immediately from Lemma~\ref{lem:deformed-ideal} which shows that $d_N(\wv{i})$ is in the ideal generated by $\sum_{j=0}^{N}\kappa_{j}x_1^{N-j}$.
\end{proof}

\begin{proposition}
For each $N>0$, the pair $(\extL{n},d_N^{\Sigma})$ is a DG-subalgebra of $(\extNH{n},d_N^{\Sigma})$.
\end{proposition}

\begin{proof}
This is immediate since the differential $d_N^{\Sigma}$ acting on $\hwv{i}$ can be expressed as a linear combination of undeformed differentials, each of which preserves the ring $\extL{n}$.
\end{proof}

\begin{theorem}
The DG-algebra $(\extL{n},d_N^{\Sigma})$ is quasi-isomorphic to the ring $H_n^{\Sigma}$ from Theorem~\ref{thm:RW}.
\end{theorem}

\begin{proof}
This follows from \cite[Lemma 11]{RoseW} and Proposition~\ref{prop:dN-restricts}.
\end{proof}

\subsection{Categorification}
\label{ss:cat}

Let $\mathbf{f}$ denote the positive part $\mathbf{U}^+(\mf{sl}_2)$ of the quantized universal enveloping algebra of $\mf{sl}_2$.  This $\mathbb{Q}(q)$-algebra is a polynomial ring in the generator $E$.  This algebra is $\N$-graded with $E$ in degree 2.  We equip the tensor product $\mathbf{f} \otimes \mathbf{f}$ with the twisted algebra structure
\[
(E^{a} \otimes E^b) (E^c \otimes E^d) = q^{-2cd} E^aE^c \otimes E^bE^d.
\]
The algebra $\mathbf{f}$ contains a subring $_{\cal{A}}\mathbf{f}$ which is the $\Z[q,q^{-1}]$-lattice generated by all products of quantum divided powers
\begin{equation} \label{eq:divided-power}
 E^{(n)} := \frac{E^n}{[n]!}.
\end{equation}
Hence, a categorification of $_{\cal{A}}\mathbf{f}$ amounts to identifying objects $\cal{E}^{(n)}$ and $\cal{E}^n$ in a graded category and lifting the divided power relation~\eqref{eq:divided-power} to an explicit isomorphism
\begin{equation}
  \cal{E}^n \cong \bigoplus_{[n]!} \cal{E}^{(n)} = \cal{E}^{(n)}\la n-1\ra \oplus \cal{E}^{(n)}\la n-3\ra \oplus \cdots \oplus \cal{E}^{(n)}\la 1-n\ra.
\end{equation}

The results from the previous section allow us to define a categorification of $_{\cal{A}}\mathbf{f}$.
Consider the graded ring
\[
\extNH{}:= \bigoplus_{n \geq 0} \extNH{n},
\]
and denote by $\extNH{}\smod$ the category of projective graded $\extNH{}$-supermodules.  Recall from Proposition \ref{prop:structureOfNH} the isomorphism $\extNH{n} \cong {\rm Mat}((n)^!_{q^2}, \extL{n})$.  One can easily show that $e_n = \und{x}^{\delta}\partial_{w_0}$ is the minimal idempotent projecting onto the lowest degree column of $\extNH{n}$.
The graded module $\extNH{n}e_n$ is the unique indecomposable projective $\extNH{n}$ up to isomorphism and grading shift.  The regular representation then decomposes into $n!$ isomorphic copies of $\extNH{n}e_n$.  Taking gradings into account, if we define
\[
\cal{E}^{(n)} := \extNH{n}e_n\la -n(n-1)/2\ra, \qquad \cal{E}^{n} := \extNH{n},
\]
then we have an isomorphism of projective left supermodules
\[
\cal{E}^n := \extNH{n} \cong \bigoplus_{[n]!} \extNH{n}e_n =: \bigoplus_{[n]!}\cal{E}^{(n)} .
\]
Hence, we have proven the following.
\begin{proposition}
There is an isomorphism of $\cal{A}$-modules
\begin{equation}
 \gamma \maps  _{\cal{A}}\mathbf{f} \to K_0(\extNH{})
\end{equation}
sending $E^{(n)}$ to the class of the indecomposable projective module $\cal{E}^{(n)}$.
\end{proposition}

There are inclusions of graded super-rings
\begin{equation}
  \iota_{n,m} \maps \extNH{n} \otimes \extNH{m} \to \extNH{n+m}
\end{equation}
given diagrammatically by placing diagrams side-by-side with those in $\extNH{n}$ appearing above $\extNH{m}$.   These inclusions give rise to induction and restriction functors
\begin{align}
 \Ind_{n,m} &\maps \left(\extNH{n} \otimes \extNH{m}\right)\smod \to \extNH{n+m}\smod, \nn\\
  \Res_{n,m} &\maps \extNH{n+m}\smod\to \left(\extNH{n} \otimes \extNH{m}\right)\smod, \nn
\end{align}
By the basis theorem ~\ref{thm:ext-sym} for $\extNH{n+m}$ it follows that the super module $\extNH{n+m}$ is a free graded left super $\extNH{n} \otimes \extNH{m}$-module.  A basis is given by the crossing diagrams in $\extNH{n+m}$ corresponding to the minimal representative of a left $\Sgp_n \times S_m$-coset in $S_{n+m}$, see for example \cite[Proposition 2.16]{KL1}.  It follows that $\Res_{n,m}$ takes projectives to projectives, and therefore descends to a map in the Grothendieck group.  Similarly, by a version of the Mackey induction-restriction theorem it follows that $\Ind_{n,m}$ also sends projectives to projectives.

Summing over all $n,m \in \Z_{\geq 0}$ these functors induce maps
\begin{align}
 [\Ind] &\maps K_0\left(\extNH{}\right) \otimes K_0\left(\extNH{}\right) \to K_0\left(\extNH{}\right), \nn\\
  [\Res] &\maps K_0\left(\extNH{}\right)\to K_0\left(\extNH{}\right) \otimes K_0\left(\extNH{}\right), \nn
\end{align}
Just as in the case of the nilHecke algebra, see \cite{KL1}, induction and restriction equip $_{\cal{A}}\mathbf{f}$ with the structure of a twisted  bialgebra and we have the following result.

\begin{theorem}
The isomorphism
\begin{equation}
 \gamma \maps  _{\cal{A}}\mathbf{f} \to K_0(\extNH{})
\end{equation}
 is an isomorphism of twisted bialgebras.
\end{theorem}


%

%

%

\bibliographystyle{amsalpha}
\bibliography{bib_extend}

\providecommand{\bysame}{\leavevmode\hbox to3em{\hrulefill}\thinspace}
\providecommand{\MR}{\relax\ifhmode\unskip\space\fi MR }
\providecommand{\MRhref}[2]{%
  \href{http://www.ams.org/mathscinet-getitem?mr=#1}{#2}
}
\providecommand{\href}[2]{#2}
\begin{thebibliography}{KLMS12}

\bibitem[AH15]{abel-hogancamp}
M.~Abel and M.~Hogancamp, \emph{Categorified {Y}oung symmetrizers and stable
  homology of torus links {II}},
  \href{http://arxiv.org/abs/1510.05330}{arXiv:1510.05330}.

\bibitem[EK12]{EK}
A.~P. Ellis and M.~Khovanov, \emph{The {H}opf algebra of odd symmetric
  functions}, Advances in Mathematics \textbf{231} (2012), no.~2, 965--999,
  \href{http://arxiv.org/abs/1107.5610}{arXiv:math.QA/1107.5610}.

\bibitem[EKL14]{EKL}
A.P. Ellis, M.~Khovanov, and A.D. Lauda, \emph{The odd nil{H}ecke algebra and
  its diagrammatics}, Int. Math. Res. Not. IMRN (2014), no.~4, 991--1062,
  \href{http://arxiv.org/abs/1111.1320}{arXiv:1111.1320}.

\bibitem[HL10]{Lau4}
A.~Hoffnung and A.~D. Lauda, \emph{Nilpotency in type {A} cyclotomic
  quotients}, Journal of Algebraic Combinatorics \textbf{32} (2010), 533--555,
  math.RT/0903.2992.

\bibitem[Kan01]{Kane}
R.~Kane, \emph{Reflection groups and invariant theory}, CMS Books in
  Mathematics/Ouvrages de Math\'ematiques de la SMC, vol.~5, Springer-Verlag,
  New York, 2001.

\bibitem[KK86]{KoKu}
B.~Kostant and S.~Kumar, \emph{The nil {H}ecke ring and cohomology of {$G/P$}
  for a {K}ac-{M}oody group {$G$}}, Adv. in Math. \textbf{62} (1986), no.~3,
  187--237.

\bibitem[KK11]{KK}
S.-J. Kang and M.~Kashiwara, \emph{Categorification of highest weight modules
  via {K}hovanov-{L}auda-{R}ouquier algebras},
  \href{http://arxiv.org/abs/1102.4677}{arXiv:1102.4677}.

\bibitem[KKO13]{KKO}
S.-J. Kang, M.~Kashiwara, and S.-J. Oh, \emph{Supercategorification of quantum
  {K}ac-{M}oody algebras}, Adv. Math. \textbf{242} (2013), 116--162,
  \href{http://arxiv.org/abs/1206.5933}{arXiv:math.RT/1206.5933},.

\bibitem[KL09]{KL1}
M.~Khovanov and A.~Lauda, \emph{A diagrammatic approach to categorification of
  quantum groups {I}}, Represent. Theory \textbf{13} (2009), 309--347,
  \href{http://arxiv.org/abs/0803.4121}{arXiv:0803.4121}.

\bibitem[KL10]{KL3}
\bysame, \emph{A diagrammatic approach to categorification of quantum groups
  {III}}, Quantum Topology \textbf{1} (2010), 1--92,
  \href{http://arxiv.org/abs/0807.3250}{arXiv:0807.3250}.

\bibitem[KL11]{KL2}
\bysame, \emph{A diagrammatic approach to categorification of quantum groups
  {II}}, Trans. Amer. Math. Soc. \textbf{363} (2011), 2685--2700,
  \href{http://arxiv.org/abs/0804.2080}{arXiv:0804.2080}.

\bibitem[KLMS12]{KLMS}
M.~Khovanov, A.~Lauda, M.~Mackaay, and M.~{S}to{\v{s}i\'c}, \emph{Extended
  graphical calculus for categorified quantum sl(2)}, Memoirs of the AMS
  \textbf{219} (2012), \href{http://arxiv.org/abs/1006.2866}{arXiv:1006.2866}.

\bibitem[KR08a]{KhR}
M.~Khovanov and L.~Rozansky, \emph{Matrix factorizations and link homology},
  Fund. Math. \textbf{199} (2008), no.~1, 1--91.

\bibitem[KR08b]{KR08b}
M.~Khovanov and L.~Rozansky, \emph{Matrix factorizations and link homology.
  {II}}, Geom. Topol. \textbf{12} (2008), no.~3, 1387--1425. \MR{2421131
  (2010g:57014)}

\bibitem[KW08]{KW1}
T.~Khongsap and W.~Wang, \emph{{H}ecke-{C}lifford algebras and spin {H}ecke
  algebras {I}: {T}he classical affine type}, Transf. Groups \textbf{13}
  (2008), 389--412,
  \href{http://arxiv.org/abs/0704.0201}{arXiv:math.RT/0704.0201}.

\bibitem[Lau08]{Lau1}
A.~D. Lauda, \emph{A categorification of quantum sl(2)}, Adv. Math.
  \textbf{225} (2008), 3327--3424,
  \href{http://arxiv.org/abs/arXiv:0803.3652}{arXiv:0803.3652}.

\bibitem[Lau12]{Lau3}
\bysame, \emph{An introduction to diagrammatic algebra and categorified quantum
  ${\mathfrak{sl}}_2$}, Bulletin Inst. Math. Academia Sinica \textbf{7} (2012),
  165--270, \href{http://arxiv.org/abs/arXiv:1106.2128}{arXiv:1106.2128}.

\bibitem[LR14]{LauR}
A.D. Lauda and H.~Russell, \emph{Oddification of the cohomology of type {$A$}
  {S}pringer varieties}, Int. Math. Res. Not. IMRN (2014), no.~17, 4822--4854,
  \href{http://arxiv.org/abs/1203.0797}{arXiv:1203.0797}.

\bibitem[LS82]{Las}
A.~Lascoux and M.~P. Sch{\"u}tzenberger, \emph{Polyn\^omes de {S}chubert}, C.
  R. Acad. Sci. Paris S\'er. I Math. \textbf{294} (1982), no.~13, 447--450.

\bibitem[LV11]{LV}
A.~D. Lauda and M.~Vazirani, \emph{Crystals from categorified quantum groups},
  Adv. Math. \textbf{228} (2011), no.~2, 803--861,
  \href{http://arxiv.org/abs/0909.1810}{arXiv:0909.1810}.

\bibitem[Man01]{Man}
L.~Manivel, \emph{Symmetric functions, {S}chubert polynomials and degeneracy
  loci}, SMF/AMS Texts and Monographs, vol.~6, American Mathematical Society,
  Providence, RI, 2001, Translated from the 1998 French original by John R.
  Swallow, Cours Sp\'ecialis\'es [Specialized Courses], 3.

\bibitem[NV16]{Vaz}
G.~Naisse and P.~Vaz, \emph{An approach to categorification of {V}erma
  modules}, \href{http://arxiv.org/abs/0905.2059}{arXiv:0905.2059}.

\bibitem[NV17a]{Vaz3}
\bysame, \emph{2-{V}erma modules and link homology},
  \href{http://arxiv.org/abs/1704.08485}{arXiv:1704.08485}.

\bibitem[NV17b]{Vaz2}
\bysame, \emph{On 2-{V}erma modules for quantum $\mf{sl}_2$},
  \href{http://arxiv.org/abs/1704.08205}{arXiv:1704.08205}.

\bibitem[Ras15]{Rasmussen}
J.~Rasmussen, \emph{Some differentials on {K}hovanov-{R}ozansky homology},
  Geom. Topol. \textbf{19} (2015), no.~6, 3031--3104.

\bibitem[Rou08]{Rou2}
R.~Rouquier, \emph{2-{K}ac-{M}oody algebras}, 2008, arXiv:0812.5023.

\bibitem[RW15]{RoseW}
D.~E.~V. Rose and P.~Wedrich, \emph{Deformations of colored sl({N}) link
  homologies via foams}, arXiv:1501.02567.

\bibitem[Sol63]{Solomon}
L.~Solomon, \emph{Invariants of finite reflection groups}, Nagoya Math. J.
  \textbf{22} (1963), 57--64. \MR{0154929}

\bibitem[Wan09]{Wang}
W.~Wang, \emph{Double affine {H}eke algebras for the spin symmetric group},
  Math. Res. Lett. \textbf{16} (2009), 1071--1085,
  \href{http://arxiv.org/abs/math/0608074}{arXiv:math.RT/0608074}.

\bibitem[Web10]{Web}
B.~Webster, \emph{Knot invariants and higher representation theory {I}:
  diagrammatic and geometric categorification of tensor products},
  \href{http://arxiv.org/abs/1001.2020}{arXiv:1001.2020}.

\bibitem[Web13]{Web5}
\bysame, \emph{Knot invariants and higher representation theory},
  \href{http://arxiv.org/abs/1309.3796}{arXiv:1309.3796}.

\bibitem[Wu12]{Wu2}
H.~Wu, \emph{Equivariant colored {$\mf{sl}(N)$}-homology for links}, J. Knot
  Theory Ramifications \textbf{21} (2012), no.~2, 1250012, 104.

\bibitem[Wu14]{Wu}
\bysame, \emph{A colored {$\mf{sl}(N)$} homology for links in {$S^3$}},
  Dissertationes Math. (Rozprawy Mat.) \textbf{499} (2014), 217.

\bibitem[WW09]{WW09}
B.~{Webster} and G.~{Williamson}, \emph{{A geometric construction of colored
  HOMFLYPT homology}}, 2009.

\bibitem[Yon11]{Yoshi}
Y.~Yonezawa, \emph{Quantum {$(\mf{sl}_n,\wedge V_n)$} link invariant and matrix
  factorizations}, Nagoya Math. J. \textbf{204} (2011), 69--123.

\end{thebibliography}
%
%

\end{document}